\documentclass[12pt]{amsart}
\usepackage{amssymb,mathrsfs}
\usepackage{eepic}
\usepackage[dvipdfm]{graphicx}
\usepackage[all]{xy}

\newtheorem{theorem}{Theorem}[section]
\newtheorem{proposition}[theorem]{Proposition}
\newtheorem{lemma}[theorem]{Lemma}
\newtheorem{corollary}[theorem]{Corollary}

\theoremstyle{definition}
\newtheorem{definition}[theorem]{Definition}

\newtheorem{example}[theorem]{Example}
\newtheorem{problem}[theorem]{Problem}

\newtheorem{conjecture}[theorem]{Conjecture}
\newtheorem{remark}[theorem]{Remark}
\newtheorem{construction}[theorem]{Construction}

\newcommand{\NN}{ \ensuremath{\mathbb{N}}}
\newcommand{\ZZ}{ \ensuremath{\mathbb{Z}}}

\newcommand{\RR}{ \ensuremath{\mathbb{R}}}

\newcommand{\rank}{\ensuremath{\mathrm{rank}}\hspace{1pt}}

\newcommand{\Hom}{\ensuremath{\mathrm{Hom}}\hspace{1pt}}
\newcommand{\Ext}{\ensuremath{\mathrm{Ext}}\hspace{1pt}}

\newcommand{\Image}{\ensuremath{\mathrm{Im}}\hspace{1pt}}

\newcommand{\aaa}{\mathbf{a}}
\newcommand{\ee}{\mathbf{e}}

\newcommand{\lk}{{\mathrm{lk}}}

\newcommand{\cost}{\mathrm{cost}}

\def\cocoa{{\hbox{\rm C\kern-.13em o\kern-.07em C\kern-.13em o\kern-.15em A}}}

\newcommand{\uu}{\mathbf{u}}
\newcommand{\vv}{\mathbf{v}}

\newcommand{\cc}{\mathbf{c}}
\newcommand{\dd}{\mathbf{d}}
\newcommand{\bb}{\mathbf{b}}
\newcommand{\ttt}{\mathbf{t}}
\newcommand{\FF}{\mathcal{F}}
\newcommand{\GG}{\mathcal{G}}
\newcommand{\res}{\mathrm{res}}
\newcommand{\mult}{\mathrm{mult}}

\newcommand{\supp}{\mathrm{supp}}
\newcommand{\im}{\mathrm{im}\ \!}
\newcommand{\coker}{\mathrm{coker}\ \!}

\newcommand{\ichi}{ \mathbf 1 }
\newcommand{\C}{\mathscr C}

\newcommand\zero{\mathbf{0}}
\newcommand\Est{{}^*\!E}
\newcommand\too{\longrightarrow}
\newcommand\gMod{{}^*\! \operatorname{Mod}}
\newcommand\ve{\varepsilon}
\newcommand\Sq{\operatorname{Sq}}

\newcommand\Db{{\mathsf D}^b}

\newcommand\Sh{\operatorname{Sh}}

\begin{document}

\title{Squarefree $P$-modules and the $\cc\dd$-index}

\author{Satoshi Murai}
\address{
Satoshi Murai,
Department of Mathematical Science,
Faculty of Science,
Yamaguchi University,
1677-1 Yoshida, Yamaguchi 753-8512, Japan.
}
\email{murai@yamaguchi-u.ac.jp}

\author{Kohji Yanagawa}
\address{
Kohji Yanagawa,
Department of Mathematics, Kansai University, Suita 564-8680, Japan.
}
\email{yanagawa@kansai-u.ac.jp}

\thanks{
The first author was partially supported by JSPS KAKENHI 22740018, and the second author was partially supported by JSPS KAKENHI 25400057.} 

\begin{abstract}
In this paper, we introduce a new algebraic concept, which we call squarefree $P$-modules.
This concept is inspired from Karu's proof of the non-negativity of the $\cc\dd$-indices of Gorenstein* posets,
and supplies a way to study $\cc\dd$-indices from the viewpoint of commutative algebra.
Indeed, by using the theory of squarefree $P$-modules,
we give several new algebraic and combinatorial results on CW-posets.
First, we define an analogue of the $\cc\dd$-index for any CW-poset and prove its non-negativity when a CW-poset is Cohen--Macaulay.
This result proves that the $h$-vector of the barycentric subdivision of a Cohen--Macaulay regular CW-complex is unimodal.
Second, we prove that the Stanley--Reisner ring of the barycentric subdivision of an odd dimensional Cohen--Macaulay polyhedral complex has the weak Lefschetz property.
Third, we obtain sharp upper bounds of the $\cc\dd$-indices of Gorenstein* posets
for a fixed rank generating function.
\end{abstract}

\maketitle

\section{Introduction}

After the beautiful proof of the upper bound theorem for triangulated spheres given by Stanley \cite{St75} and Reisner \cite{Re},
the study of Stanley--Reisner rings and its applications to $f$-vector theory have been of great interest both
in combinatorics and in commutative algebra.
In this paper, we introduce a new algebraic concept to study flag $f$-vectors of finite posets,
which we call squarefree $P$-modules,
and consider its applications.

Squarefree $P$-modules are defined as an analogue of squarefree modules \cite{Ya}
which are module-theoretic generalization of Stanley--Reisner rings.
The concept of squarefree $P$-modules is inspired from the work of Karu \cite{Ka} who proved the non-negativity of the $\cc\dd$-indices of Gorenstein* posets
by using sheaves of finite vector spaces on posets.
Indeed, we show that there is a one-to-one correspondence between the squarefree $P$-modules and the sheaves on a poset $P$,
and that squarefree $P$-modules give a way to interpret Karu's proof of the non-negativity of $\cc\dd$-indices in terms of commutative algebra.
We give the definition of squarefree $P$-modules in Section 2,
and 
study their basic algebraic properties in the first half of the paper.

In the latter half of the paper, we consider applications of squarefree $P$-modules to $f$-vector theory,
particularly to the study of $\cc\dd$-indices.
First, we quickly review the theory of $\cc\dd$-indices.
We refer the readers to \cite[Section 3]{St12} for basics on the theory of partially ordered sets.
Let $P$ be a finite partially ordered set (poset) of rank $n$ with the minimal elements $\hat 0$.
The {\em order complex} $\Delta_P$ of $P$ is the abstract simplicial complex whose faces are chains of $P \setminus \{ \hat 0\}$,
namely,
$$\Delta_P=\{ \{ \sigma_1,\dots,\sigma_k\} \subset P \setminus \{\hat 0\}: \sigma_1 < \cdots < \sigma_k\}.$$
(In this paper, we assume that all posets have the minimal element,
but we ignore it when we consider their order complexes.)
Note that the rank of $P$ is the maximal cardinality of elements of $\Delta_P$.
For a finite set $X$, we write $|X|$ for its cardinality.
For a subset $S \subset [n]=\{1,\dots,n\}$,  an element $C=\{\sigma_1,\dots,\sigma_k\} \in \Delta_P$ is called an {\em $S$-chain}
if $\{ \rank \sigma_1,\dots,\rank \sigma_k\}=S$,
where we define $\rank \sigma = \max \{ |C|: C \in \Delta_P, \max C =\sigma\}$ for $\sigma \ne \hat 0$ and $\rank \hat 0 =0$.
Let $f_S(P)$ be the number of $S$-chains of $P$.
Define $h_S(P)$ by
$$h_S(P)=\sum_{T \subseteq S} (-1)^{|S|-|T|} f_T(P).$$
The vectors $(f_S(P):S \subset [n])$ and $(h_S(P): S \subset [n])$ are called the {\em flag $f$-vector} and the {\em flag $h$-vector} of $P$ respectively.

We express the flag $h$-vector of $P$ by the homogeneous non-commutative polynomial,
called the $\aaa\bb$-index of $P$.
For a subset $S \subset [n]$, its {\em characteristic monomial}
is the non-commutative monomial 
$w_S=w_1w_2 \cdots w_n$ in variables $\aaa$ and $\bb$
defined by $w_i = \aaa$ if $i \not \in S$ and $w_i=\bb$ if $i \in S$.
The {\em $\aaa\bb$-index} of $P$ is the polynomial
$$\Psi_P(\aaa,\bb)= \sum _{S \subset [n]} h_S(P) w_S \in \ZZ\langle\aaa,\bb\rangle,$$
where $\ZZ\langle \aaa, \bb \rangle$ denotes the non-commutative polynomial ring over $\ZZ$ with the variables $\aaa$ and $\bb$.
Now we recall the definition of the $\cc\dd$-index.
We say that $P$ is Gorenstein* if $\Delta_P$ is a Gorenstein* simplicial complex (see \cite[p.\ 67]{StG}).
If $P$ is Gorenstein* then there is a non-commutative polynomial $\Phi_P(\cc,\dd) \in \ZZ \langle \cc,\dd\rangle$
in the variables $\cc$ and $\dd$ such that $\Phi_P(\aaa+\bb,\aaa\bb+\bb\aaa)=\Psi_P(\aaa,\bb)$ (see \cite[Theorem 3.17.1]{St12}).
This polynomial $\Phi_P(\cc,\dd)$ is called the {\em $\cc\dd$-index} of $P$.
Note that $\Phi_P(\cc,\dd)$ is homogeneous of degree $n$ in the grading $\deg \cc=1$ and $\deg \dd=2$.
In this paper, to simplify the notation,
we regard $\Phi_P(\cc,\dd)$ as a polynomial in $\ZZ\langle \aaa,\bb\rangle$ by
the identifications $\cc=\aaa+\bb$ and $\dd=\aaa\bb+\bb\aaa$.

The $\cc\dd$-index has two important properties.
First, it efficiently encodes the flag $f$-vectors of Gorenstein* posets.
Indeed, flag $f$-vectors of posets of rank $n$ have $2^n$ entries, however,
the $\cc\dd$-polynomial of degree $n$ is a linear combination of  the $(n+1)$th Fibonacci number $F_{n+1}$ ($F_1=F_2=1, F_{k+2}=F_{k+1} + F_k$),
which is much smaller than $2^n$,
of monomials 
and it is known that the existence of the $\cc\dd$-index describes all linear equations satisfied by the flag $f$-vectors of Gorenstein* posets \cite{BB,BK}.
Another important property of the $\cc\dd$-index of a Gorenstein* poset is its non-negativity,
which was proved by Karu \cite{Ka}.

Our first result is an extension of the notion of the $\cc\dd$-index to CW-posets.
Let $P$ be a finite poset with the minimal element $\hat 0$.
We say that $P$ is a {\em quasi CW-poset} if, for every $\sigma \in P \setminus \{\hat 0\}$, the poset $ \partial \sigma = \{\tau \in P: \tau < \sigma\}$
is Gorenstein*,
and that $P$ is a {\em CW-poset} if, for every $\sigma \in P \setminus \{\hat 0\}$, the geometric realization of $\Delta_{\partial \sigma}$ is homeomorphic to a sphere.
Note that CW-posets and Gorenstein* posets are quasi CW-posets.
Also, since a finite poset is a CW-poset if and only if it is the face poset of a finite regular CW-complex \cite{Bj},
considering CW-posets are equivalent to considering finite regular CW-complexes.
A finite poset $P$ is said to be {\em Cohen--Macaulay} if the simplicial complex $\Delta_P$ is a Cohen--Macaulay simplicial complex (see \cite[p.\ 58]{StG}).

\begin{theorem}
\label{main1}
Let $P$ be a quasi CW-poset of rank $n$.
There are unique $\cc\dd$-polynomials $\Phi^\dd,\Phi^\aaa,\Phi^\bb \in \ZZ \langle \cc, \dd \rangle$ such that
\begin{align}
\label{excd} 
\Psi_P(\aaa,\bb)= \Phi^\dd \cdot \dd+\Phi^\aaa \cdot \aaa + \Phi^\bb \cdot \bb.
\end{align}
Moreover, if $P$ is Cohen--Macaulay then all the coefficients in $\Phi^\dd,\Phi^\aaa$ and $\Phi^\bb$ are non-negative.
\end{theorem}

Since $\Phi^\dd,\Phi^\aaa,\Phi^\bb$  are homogeneous of degrees $n-2,n-1$ and $n-1$ respectively,
Theorem \ref{main1} gives a way to express flag $f$-vectors of
CW-posets by $(n+2)$th Fibonacci number of integers.
We prove in Proposition \ref{4.2} that the existence of the expression \eqref{excd} indeed describes all linear equations satisfied by the flag $f$-vectors of (Cohen--Macaulay) CW-posets.
Also, the non-negativity statement for Cohen--Macaulay quasi CW-posets implies an interesting result on ordinal $h$-vectors.
For a poset $P$ of rank $n$,
the {\em $h$-vector} $h(\Delta_P)=(h_0,h_1,\dots,h_n)$ of $\Delta_P$ is given by $h_i=\sum_{S \subset [n],|S|=i} h_S (P)$.
We say that a vector $(h_0,h_1,\dots,h_n) \in \ZZ^{n+1}$ is {\em unimodal} if
$h_0 \leq h_1 \leq \cdots \leq h_p \geq \cdots \geq h_n$ for some integer $0 \leq p \leq n$.
By using Theorem \ref{main1} we prove

\begin{corollary}
\label{main2}
If $P$ is a Cohen--Macaulay quasi CW-poset then the $h$-vector of $\Delta_P$ is unimodal.
\end{corollary}

Recall that if $P$ is a CW-poset, then $\Delta_P$ is combinatorially isomorphic to the barycentric subdivision
of a regular CW-complex corresponding to $P$.
It was proved by Brenti and Welker \cite[Corollary 3]{BW}
that the $h$-vector of the barycentric subdivision of a Cohen--Macaulay Boolean cell complex is unimodal.
Corollary \ref{main2} says that this unimodality result holds in the level of regular CW-complexes.

In $f$-vector theory, once we have an unimodal sequence $h_0 \leq \cdots \leq h_p \geq \cdots \geq h_n$,
it is natural to ask if the sequence $(h_0, h_1-h_0,\dots, h_p-h_{p-1})$ has a nice property.
We study this problem for the $h$-vectors of the barycentric subdivisions of polyhedral complexes.
We say that a CW-poset $P$ is {\em of polyhedral type} if,
for every $\sigma \in P\setminus \{\hat 0\}$,
the subposet $\langle \sigma \rangle =\{ \tau \in P: \tau \leq \sigma\}$ is the face poset of the boundary of a convex polytope.
Let $\lfloor x \rfloor$ denote the integer part of $x \in \mathbb Q$.

\begin{theorem}
\label{main3}
Let $P$ be a Cohen--Macaulay CW-poset of polyhedral type having rank $n$ and let $h(\Delta_P)=(h_0,h_1,\dots,h_n)$ be the $h$-vector of $\Delta_P$.
Then the vector $(h_0,h_1-h_0,\dots,h_{\lfloor \frac n 2\rfloor}-h_{\lfloor \frac n 2\rfloor-1})$ is the $f$-vector of a simplicial complex.
\end{theorem}

To prove the above theorem, we study the weak Lefschetz property (WLP for short) of squarefree $P$-modules.
See Section 6 for the definition of the WLP.
Indeed, we prove that the order complex of a Cohen--Macaulay CW-poset of polyhedral type has the WLP over $\RR$ if its rank is even (Corollary \ref{5.3}).
This result gives a partial solution to the conjecture of Kubitzke and Nevo \cite[Conjecture 4.12]{KN}
who conjectured that the barycentric subdivision of a Cohen--Macaulay simplicial complex has the WLP.

Our final result is about upper bounds of the $\cc\dd$-indices of Gorenstein* posets for a fixed rank generating function.
To state the result, we introduce a way to describe $\cc\dd$-monomials with certain subsets.
Let  $\mathcal A_n$ be the set of subsets of $[n-1]$ which contains no consecutive integers, namely,
$$\mathcal A_n=\{S \subset [n-1]: \{i,i+1\} \not \subset S \mbox{ for }i=1,2,\dots,n-2\}.$$
Let $\mathcal B_n$ be the set of $\cc\dd$-monomials of degree $n$.
Then there is a bijection $\kappa_n : \mathcal B_n \to \mathcal A_n$ defined by
$$ \kappa_n (\cc^{s_0} \dd \cc^{s_1} \dd \cdots \dd \cc^{s_k}) = \{ s_0+1,s_0+s_1 +3,\dots,s_0+ \cdots + s_{k-1}+2k-1\}.$$
For a $\cc\dd$-polynomial $\Phi=\sum_{v \in \mathcal B_n} \alpha_v v$ of degree $n$, where $\alpha_v \in \ZZ$,
and for any $S \in \mathcal A_n$,
let $\alpha_S(\Phi)= \alpha_{\kappa^{-1}(S)}$.
Thus, if we write $\alpha_S=\alpha_{S}(\Phi)$ and $v^S=\kappa^{-1}_n (S)$, then we have $\Phi= \sum_{S \in \mathcal A_n} \alpha_S  v^S$.
For example, a $\cc\dd$-polynomial of degree $4$ will be written in the form 
$\Phi= \alpha_\emptyset \cc^4+ \alpha_{\{1\}} \dd\cc^2 + \alpha_{\{2\}} \cc\dd\cc+\alpha_{\{3\}} \cc^2 \dd+ \alpha_{\{1,3\}} \dd^2.$
For a Gorenstein* poset $P$,
we define $\alpha_S(P)=\alpha_S(\Phi_P(\cc,\dd))$.
We prove the following bound for $\alpha_S(P)$.

\begin{theorem}
\label{main4}
If $P$ is a Gorenstein* poset of rank $n$,
then $\alpha_S(P) \leq \prod_{i \in S} \alpha_{\{i\}}(P) $ for all $S \in \mathcal A_n$.
\end{theorem}

The upper bounds in Theorem \ref{main4} are sharp.
Indeed, we prove in Proposition \ref{6.5} that for any sequence $\alpha_1,\dots,\alpha_{n-1}$ of non-negative integers  there is a Gorenstein* poset $P$ of rank $n$
such that $\alpha_S(P) = \prod_{i \in S} \alpha_i $ for all $S \in \mathcal A_n$.
Also, since knowing $\alpha_{\{1\}}(P),\dots,\alpha_{\{n-1\}}(P)$ is equivalent to knowing $f_{\{1\}}(P),\dots,f_{\{n\}}(P)$ (see Section 7),
Theorem \ref{main4} gives sharp upper bounds of the $\cc\dd$-indices of Gorenstein* posets for a fixed rank generating function,
and therefore gives sharp upper bounds of the flag $h$-vectors of Gorenstein* posets for a fixed rank generating function.

This paper is organized as follows:
In Section 2, we define squarefree $P$-modules and study their basic algebraic properties.
In Section 3 and 4, we study homological properties of squarefree $P$-modules.
In Section 5, we translate Karu's proof of the non-negativity of $\cc\dd$-indices in terms of squarefree $P$-modules,
and prove Theorem \ref{main1} and Corollary \ref{main2}.
In Section 6, we study the weak Lefschetz property of Cohen--Macaulay squarefree $P$-modules,
and prove Theorem \ref{main3}.
In Section 7, upper bounds of the $\cc\dd$-indices of Gorenstein* posets for a fixed rank generating function are studied.
In Section 8, we present some open problems that arose during this research.

\section{Squarefree $P$-modules}

Throughout the paper,
we assume that every poset is finite and has the minimal element $\hat 0$.
For a poset $P$ and $\sigma \in P$,
we write $\widehat P=P\setminus \{\hat 0\}$,
$\langle \sigma \rangle =\{\tau \in P: \tau \leq \sigma\}$
and $\partial \sigma = \langle \sigma \rangle \setminus \{\sigma \}$.
We denote by $\NN$ the set of non-negative integers.

In the rest of this section,
we fix a poset $P$ of rank $n$.

\subsection*{Squarefree $P$-modules}
Let $R=K[x_\sigma: \sigma \in \widehat P]$
be the polynomial ring over a field $K$ whose variables are indexed by the elements of $\widehat P$.
We consider the $\NN^{|\widehat P|}$-grading of $R$
by defining that the degree of $x_\sigma$ is $\ee_\sigma \in \NN^{|\widehat P|}$,
where $\{\ee_\sigma: \sigma \in \widehat P\}$ is the basis of $\NN^{|\widehat P|}$.
For $\uu = \sum_{\sigma \in \widehat P} u_\sigma \ee_\sigma \in \NN^{|\widehat P|}$,
let $\supp(\uu)=\{ \sigma \in \widehat P: u_\sigma \ne 0\}$.
For an $\NN^{|\widehat P|}$-graded $R$-module $M$,
let $M_\uu$ be its graded component of degree $\uu \in \NN^{|\widehat P|}$.

\begin{definition}
\label{2.1}
A {\em squarefree $P$-module} (over $K$)
is a finitely generated $\NN^{|\widehat P|}$-graded $R$-module $M$
satisfying the following two conditions:
\begin{itemize}
\item[(a)] For any $\uu \in \NN^{|\widehat P|}$ with $\supp(\uu) \not \in \Delta_P$, one has $M_{\uu}=0$.
\item[(b)] For any $\uu \in \NN^{|\widehat P|}$ with $\supp(\uu) \in \Delta_P$ and for any $\tau \in P$,
the multiplication
$$\times x_\tau: M_\uu \to M_{\uu+\ee_\tau}$$
is bijective if $\supp(\uu+\ee_\tau) \in \Delta_P$
and $\tau \leq \max(\supp(\uu))$.
\end{itemize}
\end{definition}

A typical example of a squarefree $P$-module is the Stanley--Reisner ring of $\Delta_P$.
Recall that, for a finite abstract simplicial complex $\Delta$ with the vertex set $V$, 
its {\em Stanley--Reisner ring} (over $K$) is the ring
$$K[\Delta] = K[x_v : v \in V] /( \textstyle \prod_{v \in F} x_v : F \subset V, F \not \in \Delta).$$
For simplicity, we write $K[Q]=K[\Delta_Q]$ for a finite poset $Q$.
The Stanley--Reisner ring $K[P]$ is a squarefree $P$-module
since $(K[P])_\uu$ is the $1$-dimensional $K$-vector space spanned by the monomial $\prod_{\sigma \in \widehat P} x_\sigma^{u_\sigma}$ if $\supp(\uu) \in \Delta_P$ and is zero if $\supp(\uu) \not \in \Delta_P$.
Also, it is easy to see that,
if $Q$ is an order ideal of $P$,
that is, if $Q$ is a subposet of $P$ satisfying that $\sigma \in Q$ and $\tau < \sigma$ imply $\tau \in Q$,
then $K[Q]$ is a squarefree $P$-module (by regarding $K[Q]$ as an $R$-module).
In particular, for any $\sigma \in \widehat P$, $K[\langle \sigma \rangle]$ and $K[\partial \sigma]$ are squarefree $P$-modules.

In this paper, we say that a poset $Q$ is {\em Cohen--Macaulay} (resp.\ {\em Gorenstein}*) over a field $K$
if $K[Q]$ is Cohen--Macaulay (resp.\ Gorenstein*).
Also, when we consider the Cohen--Macaulay property (or Gorenstein* property)
we skip the condition on a field $K$ if it is arbitrary.
Note that if $Q$ is Cohen--Macaulay over some field then it is Cohen--Macaulay over $\RR$.

In the rest of this section, we discuss basic properties of squarefree $P$-modules.

\subsection*{Maps of squarefree $P$-modules}
A {\em map of squarefree $P$-modules} is 
a degree preserving $R$-homomorphism between squarefree $P$-modules.

\begin{lemma}
\label{2.2}
If $\varphi: N \to M$ is a map of squarefree $P$-modules,
then $\ker \varphi, \im \varphi$ and $\coker \varphi$ are squarefree $P$-modules.
\end{lemma}

\begin{proof}
It is clear that $\ker \varphi, \im \varphi$ and $\coker \varphi$ satisfy condition (a) of squarefree $P$-modules.
To see that they also satisfy condition (b),
consider the following commutative diagram
\begin{eqnarray*}
\begin{array}{ccccccccc}
0\longrightarrow &0 &\longrightarrow &(\ker \varphi)_\uu &\longrightarrow &N_\uu &\stackrel \varphi \longrightarrow& M_\uu\ \ \ \  \smallskip\\
& \downarrow  && \ \ \ \downarrow \times x_\tau && \ \ \ \downarrow \times x_\tau && \ \ \ \downarrow \times x_\tau\smallskip\\
0\longrightarrow &0 &\longrightarrow &(\ker \varphi)_{\uu+\ee_\tau} &\longrightarrow & N_{\uu+\ee_\tau} &\stackrel \varphi\longrightarrow& M_{\uu+\ee_\tau},
\end{array}
\end{eqnarray*}
The above diagram and the five lemma imply that $\ker \varphi$ satisfies condition (b).
The proofs for $\im \varphi$ and $\coker \varphi$ are similar.
\end{proof}

\subsection*{Krull dimensions}
Condition (b) of squarefree $P$-modules is equivalent to the following condition:
\begin{itemize}
\item[(b')] For any $ \sigma \in \widehat P$ and for any monomial $x^\uu= \prod_{\rho \in \widehat P} x_\rho^{u_\rho} \in R$
which is non-zero in $K[\langle \sigma \rangle]$, the multiplication
$\times x^\uu: M_{\ee_\sigma} \to M_{\ee_\sigma + \uu}$
is bijective.
\end{itemize}
This implies a useful decomposition formula of squarefree $P$-modules.
For convention,
we write $\ee_{\hat 0}=\zero$ for the zero vector in $\NN^{|\widehat P|}$ and write $K[\langle \hat 0 \rangle]=K[\{ \hat 0 \}]=K$.

\begin{lemma}
\label{2.3}
If $M$ is a squarefree $P$-module, then as $K$-vector spaces one has
$$M=\bigoplus_{\sigma \in P}\left( M_{\ee_\sigma} \otimes_K K[\langle \sigma \rangle]\right).$$
\end{lemma}

\begin{proof}
Let $N=\bigoplus_{\sigma \in P}\left( M_{\ee_\sigma} \otimes_K K[\langle \sigma \rangle]\right)$.
We claim that $N_\uu \cong M_\uu$ for all $\uu \in \NN^{|\widehat P|}$.
The claim is obvious when $\uu$ is the zero vector.
Also, by condition (a) of squarefree $P$-modules, we may assume $\supp(\uu) \in \Delta_P$.
Let $\uu \in \NN^{|\widehat P|}$ with $\supp(\uu)=\{\sigma_1,\dots,\sigma_k\} \in \Delta_P$, where $\sigma_k = \max (\supp(\uu))$.
Then we have
$$N_\uu = (M_{\ee_{\sigma_k}} \otimes_K K[ \langle \sigma_k \rangle])_\uu \cong M_{\ee_{\sigma_k}} \cong M_{\ee_{\sigma_k} + (\uu -\ee_{\sigma_k})} = M_\uu$$
as desired,
where the third equation follows from condition (b') of squarefree $P$-modules.
\end{proof}

The above lemma determines the Krull dimension of a squarefree $P$-module.
Recall that the {\em (Krull) dimension} $\dim M$ of a finitely generated graded $R$-module $M$ is the minimal number $k$ such that
there is a sequence $\theta_1,\dots,\theta_k \in R$ of positive degrees such that $\dim_K (M/((\theta_1,\dots,\theta_k)M)) < \infty$.

\begin{corollary}
\label{2.4}
If $M$ is a squarefree $P$-module, then the Krull dimension of $M$ is $\max \{ \rank \sigma : M_{\ee_\sigma} \ne 0\}$.
\end{corollary}

\begin{proof}
Since the Krull dimension of a graded $R$-module is  equal to the degree of its Hilbert polynomial plus one \cite[Theorem 4.1.3]{BH}, we have
$$\dim M= \dim \left(\bigoplus_{\sigma \in P} M_{\ee_\sigma} \otimes_K K[\langle \sigma \rangle] \right)
= \max \{ \dim K[\langle \sigma \rangle] :M_{\ee_\sigma} \ne 0\}.$$
Then the desired equation follows since $\dim K[\langle \sigma \rangle]= \rank \sigma$ by
\cite[II Theorem 1.3]{StG}.
\end{proof}

\subsection*{Hilbert series and flag $h$-vectors}
Squarefree $P$-modules have a natural $\NN^n$-graded structure
defined by $\deg x_\sigma = \ee_{\rank \sigma} \in \NN^n$,
where $\ee_i$ denotes the $i$th unit vector of $\NN^n$.
Let $M$ be a squarefree $P$-module of dimension $d$.
By the above $\NN^n$-grading, $M$ is actually $\NN^d$-graded
since Corollary \ref{2.4} says that $M_\uu=0$ for all $\uu \in \NN^{|\widehat P|}$ with $\rank(\max (\supp(\uu))) >d$.
The {\em $(\NN^d$-graded$)$ Hilbert series} of $M$ is the formal power series
$$H_M(t_1,\dots,t_d)= \sum _{\vv \in \NN^d} (\dim_K M_\vv) \ttt^\vv$$
where $\ttt^\vv=t_1^{v_1} \cdots t_d^{v_d}$ for $\vv=(v_1,\dots,v_d) \in \NN^d$.

For $S \subset [d]$,
let $\ee_S= \sum_{i \in S} \ee_i \in \NN^d$
and $\ttt^S=\prod_{i \in S} t_i$.
For a squarefree $P$-module $M$ of dimension $d$,
we define its {\em flag $f$-vector} $(f_S(M): S \subset [d])$ and the {\em flag $h$-vector} $(h_S(M): S \subset [d])$ by
$$f_S(M)= \dim_K M_{\ee_S}$$
and 
$$h_S(M)= \sum_{T \subset S} (-1)^{|S|-|T|} f_T(M),$$
where $f_\emptyset(M)=h_\emptyset(M) = \dim_K M_\zero$.
Note that if $M=K[P]$, then $d=n$ and we have $f_S(K[P])=f_S(P)$ and $h_S(K[P])=h_S(P)$ for all $S \subset [n]$.

\begin{lemma}
\label{2.5}
If $M$ is a squarefree $P$-module of dimension $d$, then
$$H_M(t_1,\dots,t_d)= \frac {\sum_{S \subset [d]} h_S(M) \ttt^S} {(1-t_1)(1-t_2) \cdots (1-t_d)}.$$
\end{lemma}

\begin{proof}
Since the multiplication $\times x_\tau : M_\uu \to M_{\uu + \ee_\tau}$ is bijective if $\tau \in \supp(\uu)$,
we have the decomposition
$$M \cong \bigoplus_{C \in \Delta_P, \atop \rank(\max(C))\leq d} 
\left( M_{\left(\sum_{\sigma \in C} \ee_\sigma\right)}\otimes_K K[x_\sigma: \sigma \in C] \right)$$
as $K$-vector spaces,
where $M_{(\sum_{\sigma \in C} \ee_\sigma)}=M_\zero$ and $K[x_\sigma: \sigma \in C]=K$ if $C=\emptyset$.
Then
\begin{align*}
H_M(t_1,\dots,t_d)&= \sum_{C \in \Delta_P, \atop \rank(\max(C)) \leq d} \left(\dim_K M_{\left(\sum_{\sigma \in C} \ee_\sigma\right)} \right) \cdot \frac {\prod_{\sigma \in C} t_{\rank \sigma}} {\prod_{\sigma \in C} (1-t_{\rank \sigma})}\\
&= \sum_{S \subset [d]} f_S(M) \cdot \frac {\ttt^S  \cdot \prod_{i \in [d]\setminus S} (1-t_i)} {(1-t_1)(1-t_2) \cdots (1-t_d)}\\
&= \frac  {\sum_{S \subset [d]} h_S(M) \ttt^S} {(1-t_1)(1-t_2)\cdots(1-t_d)},
\end{align*}
as desired.
\end{proof}

For a squarefree $P$-module $M$ of dimension $d$,
the {\em $\aaa\bb$-index} of $M$ is the polynomial $\Psi_M(\aaa,\bb)= \sum_{S \subset [d]} h_S(M) w_S$,
where $w_S$ is the characteristic monomial of $S$ defined in the introduction.
The next lemma gives another way to express the flag $h$-vector of $M$.

\begin{lemma}
\label{2.6}
If $M$ is a squarefree $P$-module of dimension $d$, then
$$\Psi_M(\aaa,\bb)= (\dim_K M_\zero) (\aaa-\bb)^d + \sum_{\sigma \in \widehat P} (\dim_K M_{\ee_\sigma}) \Psi_{\partial \sigma}(\aaa,\bb) \cdot \bb (\aaa-\bb)^{d-\rank \sigma}.$$
\end{lemma}

\begin{proof}
Observe that $K[\langle \sigma \rangle]$ is the polynomial ring over $K[\partial \sigma]$ with the variable $x_\sigma$,
that is, $K[\langle \sigma \rangle]=K[\partial \sigma][x_\sigma]$.
Hence, for any $\sigma \in \widehat P$ with $\rank \sigma=r \leq d$, we have
$$ H_{K[\langle \sigma \rangle]}(t_1,\dots,t_d)= \frac {\sum_{ S \subset [r-1]} h_S(\partial \sigma) \ttt^S} {(1-t_1)(1-t_2) \cdots (1-t_r)}.$$
(Here we identify $H_{K[\langle \sigma \rangle]}(t_1,\dots,t_r)$ and $H_{K[\langle \sigma \rangle]}(t_1,\dots,t_d)$.)
Then, by Lemma \ref{2.3}, we have
\begin{align*}
H_M(t_1,\dots,t_d) 
&= (\dim_K M_\zero) + \sum_{\sigma \in\widehat P} (\dim_K M_{\ee_{\sigma}}) t_{\rank \sigma} \cdot H_{K[\langle \sigma \rangle]} (t_1,\dots,t_d)\\
&= \frac {\sum_{\sigma \in P} (\dim_K M_{\ee_\sigma})t_{\rank \sigma} \left(\sum_{S \subset [\rank \sigma -1]} h_S(\partial \sigma) \ttt^S\right) \prod _{k > \rank \sigma} (1-t_k)} {(1-t_1)(1-t_2) \cdots (1-t_d)},
\end{align*}
where we consider $t_{\rank \hat 0} =1$.
Then, by translating the numerator of the above formula into an $\aaa\bb$-polynomial,
we obtain the desired formula.
\end{proof}

\subsection*{Sheaves on posets and squarefree $P$-modules}
Here we discuss relations between sheaves on $P$ and squarefree $P$-modules.
A {\em sheaf} $\FF$ (of finite $K$-vector spaces) on $P$ consists of the data
\begin{itemize}
\item A finite $K$-vector space $\FF_\sigma$ for each $\sigma \in P$, called the {\em stalk} of $\FF$ at $\sigma$.
\item Linear maps $\res_\tau^\sigma: \FF_\sigma \to \FF_\tau$ for all $\sigma > \tau$ in $P$,
called the {\em restriction maps},
satisfying $\res^\tau_\rho \circ \res^\sigma_\tau= \res^\sigma_\rho$
for all $\sigma > \tau > \rho$ in $P$.
\end{itemize}
A {\em map of sheaves} $\FF \to \GG$ is a collection of linear maps $\FF_\sigma \to \GG_\sigma$
commuting with the restriction maps.
Note that the formal definition of sheaves is more complicated but it is equivalent to the above one.

Let $M$ be a squarefree $P$-modules.
We can construct a sheaf $\FF^M$ on $P$ as follows:
For $\sigma > \tau$ in $P$,
we define the map $\mult_\tau^\sigma: M_{\ee_\tau} \to M_{\ee_\sigma}$ by the composition
\begin{align}
\label{mult}
M_{\ee_\tau} \stackrel {\times x_\sigma} \longrightarrow M_{\ee_\sigma+\ee_\tau} \stackrel {\ \ (\times x_\tau)^{-1}} \longrightarrow M_{\ee_\sigma},
\end{align}
where $(\times x_\tau)^{-1}$ is the inverse map of the bijection $\times x_\tau: M_{\ee_\sigma} \to M_{\ee_\sigma + \ee_\tau}$.
Then, it is straightforward that these maps satisfy
$$\mult^\sigma_\tau \circ \mult^\tau_\rho=\mult_\rho^\sigma$$
for all $\sigma > \tau > \rho$ in $P$.
We define the sheaf $\FF^M$ on $P$ by $\FF^M_\sigma =(M_{\ee_\sigma})^*$  for all $\sigma \in P$
and $\res_\tau^\sigma=(\mult_\tau^\sigma)^*$ for all $\sigma > \tau$ in $P$,
where $*$ denotes the $K$-dual.

Conversely, from a sheaf $\FF$ on $P$, we can define the squarefree $P$-module $M(\FF)$ as follows:
As graded $K$-vector spaces, we define
$$M(\FF)= \bigoplus_{\sigma \in P} (\FF_\sigma)^* \otimes_K K[\langle \sigma \rangle],$$
where we consider that each element of $(\mathcal F_\sigma)^*$ has degree $\ee_\sigma$.
Then we define the multiplication structure 
by the following rule.
For $\rho \in \widehat P$ and $m \otimes f \in (\FF_\sigma)^* \otimes K[\langle \sigma \rangle]$,
we define
\begin{align*}
x_\rho \cdot (m\otimes f)=
\begin{cases}
(\res^\rho_\sigma)^*(m) \otimes  x_\sigma f\ \ (\in (\FF_\rho)^* \otimes K[\langle \rho \rangle]), & \mbox{ if }\rho >\sigma,\\
m \otimes  x_\rho f\ \ (\in (\FF_\sigma)^* \otimes K[\langle \sigma \rangle]), & \mbox{ if }\rho \leq \sigma,\\
0, & \mbox{ otherwise.}
\end{cases}
\end{align*}

It is easy to see that $M(\FF^N) \cong N$ and $\FF^{M(\mathcal G)} \cong \mathcal G$ for any squarefree $P$-module $N$
and for any sheaf $\mathcal G$ on $P$.
Later in Section 4, we will discuss the anti-equivalence between the category 
of squarefree $P$-modules and that of sheaves on $P$, which is  given by the correspondences 
$M \mapsto \FF^M$ and $\FF \mapsto M(\FF)$.

\section{Homological properties of squarefree $P$-modules}

In this section, we study homological properties of squarefree $P$-modules when $P$ is a quasi CW-poset.
In Sections 3--5, 
we fix a quasi CW-poset $P=\bigcup_{i=0}^n P_i$ of rank $n$,
where $P_i=\{ \sigma \in P: \rank \sigma=i\}$,
and let $R=K[x_\sigma: \sigma \in \widehat P]$.
We consider the $\ZZ^{|\widehat P|}$-grading of $R$ instead of $\NN^{|\widehat P|}$-grading to deal with modules having negative graded components.

\subsection*{Augmented oriented chain complexes}
We say that an element $\sigma \in P$ {\em covers } $\tau \in P$
if $\sigma >\tau$ and $\rank \sigma= \rank \tau +1$.
Recall that, since $P$ is a quasi CW-poset, for all $\sigma>\rho$ in $P$
with $\rank \sigma =\rank \rho +2$, there are exactly two elements $\tau_1,\tau_2$ with $\sigma > \tau_i > \rho$. 
An {\em incidence function} $\varepsilon$ of $P$ is a function $\varepsilon : P \times P \to K$ satisfying the following conditions
\begin{itemize}
\item[(i)] $\varepsilon(\sigma,\tau) \ne 0$ if and only if $\sigma$ covers $\tau$.
\item[(ii)] for cover relations $\sigma>\tau_1>\rho$ and $\sigma>\tau_2>\rho$ with $\tau_1 \ne \tau_2$, one has
$$\varepsilon(\sigma,\tau_1)\varepsilon(\tau_1,\rho)+\varepsilon(\sigma,\tau_2)\varepsilon(\tau_2,\rho)=0.$$
\end{itemize}
For every quasi CW-poset, its incidence function exists and is unique 
in a certain sense
(i.e., in the sense that the augmented oriented chain complex described below is independent of the choice of an incidence function up to isomorphism of complexes).
Indeed, for a CW-poset $P$, an incidence function of $P$ coincides with that of the corresponding regular CW-complex,
and the existence and the uniqueness are standard in combinatorial topology (see e.g., \cite[V Theorem 4.2]{LW} or \cite[IV Theorem 7.2]{Ma}).
On the other hand, the corresponding statement for quasi CW-posets is obtained by the same proof
since these results for finite regular CW-complexes works even if we allow a closed cell to be the cone of a homology sphere
(here a homology sphere means a space which is homeomorphic to a Gorenstein* simplicial complex),
and quasi CW-posets are the face posets of such {\em generalized} regular CW-complexes.

By using an incidence function $\varepsilon$ of $P$,
we define the {\em augmented oriented chain complex} $\C^P_\bullet$ of $P$
as the complex
$$\C_\bullet^P :  0 \longrightarrow  \C_{n-1}^P \stackrel{\partial} \longrightarrow  \C_{n-2}^P \stackrel{\partial} \longrightarrow
\cdots \longrightarrow  \C_0^P \stackrel{\partial} \longrightarrow \C_{-1}^P \longrightarrow 0$$
where $\C_i^P = \bigoplus_{\sigma \in P_{i+1}} K \cdot \sigma$ is the $K$-vector space with basis $P_{i+1}$
and where $\partial (\sigma)=\sum_{\tau \in P_{i}, \tau < \sigma} \varepsilon(\sigma,\tau) \tau$ for $\sigma \in P_{i+1}$. 

Let $M$ be a squarefree $P$-module.
We define the {\em augmented oriented chain complex} of $M$ (or $\mathcal F^M$)
$$\C_\bullet^M :  0 \longrightarrow  \C_{n-1}^M \stackrel{\partial} \longrightarrow \C_{n-2}^M \stackrel{\partial} \longrightarrow
\cdots \longrightarrow  \C_0^M \stackrel{\partial} \longrightarrow \C_{-1}^M \longrightarrow 0$$
by $\C_i^M  = \bigoplus_{\sigma \in P_{i+1}}  \mathcal F^M_\sigma \otimes_K (K\cdot\sigma)$ and 
$\partial (\mu \otimes \sigma)=\sum_{\tau \in P_{i}, \tau < \sigma} \res^\sigma_\tau(\mu) \otimes \varepsilon(\sigma,\tau) \tau$ for $\sigma \in P_{i+1}$
and $\mu \in \FF^M_\sigma=(M_{\ee_\sigma})^*$.

\subsection*{Karu complexes}
Augmented oriented chain complexes can be naturally extended to complexes of squarefree $P$-modules.
For a squarefree $P$-module $M$,
we define the complex
$$\mathcal L_\bullet^M : 0 \longrightarrow  \mathcal L_{n}^M \stackrel{\widetilde \partial} \longrightarrow  \mathcal L_{n-1}^M \stackrel{\widetilde \partial} \longrightarrow
\cdots \longrightarrow  \mathcal L_1^M \stackrel{\widetilde \partial} \longrightarrow \mathcal L_{0}^M \longrightarrow 0$$
by
$$\mathcal L_i^M= \bigoplus_{\sigma \in P_{i}} \mathcal F_{\sigma}^M \otimes_K (K[\langle \sigma \rangle]\cdot \sigma)$$
(here we consider that elements of $\mathcal F_\sigma^M$ have degree 0) and by
$$
\widetilde \partial( \mu \otimes f \sigma)= \sum_{\tau \in P_{i} \atop \tau < \sigma} \res^\sigma_\tau (\mu) \otimes  \varepsilon(\sigma,\tau) \pi_{\sigma,\tau}(f) \tau
$$
for $\mu \otimes f \sigma \in \mathcal F_\sigma^M \otimes_K (K[\langle \sigma \rangle]\cdot \sigma)$ with $\rank \sigma=i+1$,
where $\pi_{\sigma,\tau}$ is a natural surjection $K[\langle \sigma \rangle] \twoheadrightarrow K[\langle \tau \rangle]$.
We call $\mathcal L^M_\bullet$ the {\em Karu complex} of $M$.

For a $\ZZ^{|\widehat P|}$-graded $R$-module $N$ and $\uu \in \ZZ^{|\widehat P|}$,
we write $N(-\uu)$ for the graded module $N$ with grading shifted by
$\uu$.
The Karu complex has the following important property.

\begin{theorem}
\label{3.4}
For a squarefree $P$-module $M$, one has
$$H_i(\mathcal L_\bullet^M) \cong \Ext_R^{|\widehat P|-i} (M, R(-\ichi))$$
for $i=0,1,\dots,n$,
where $\ichi =(1,1,\dots,1) \in \ZZ^{|\widehat P|}$.
\end{theorem}

We will prove the above theorem in the next section
since it requires preparation.

\begin{corollary}
\label{extSqP}
If $M$ is a squarefree $P$-module, then so is $\Ext_R^i(M, R(-\ichi))$ for all $i$. 
\end{corollary}

\begin{proof}
Observe that $\dim M \leq n$ and $\Ext_R^i(M,R(-\ichi))=0$ for $i< |\widehat P|-\dim M$.
Since ${\mathcal L}^M_\bullet$ is a complex of squarefree $P$-modules, its homologies are also squarefree $P$-modules by Lemma~\ref{2.2}.
So the assertion follows from Theorem~\ref{3.4}. 
\end{proof}

\subsection*{Cohen--Macaulay criterion}
For $\sigma \in P$, 
the poset $\lk_P(\sigma) =\{\tau \in P: \tau \geq \sigma\}$ is called the {\em link of $\sigma$} in $P$
and the poset $\cost_P(\sigma) = \{\tau \in P : \tau \not \geq \sigma\}$ is called the {\em contrastar of $\sigma$} in $P$.
Note that $\cost_P(\sigma)$ is an order ideal of $P$ and $\lk_P(\sigma)$ is a quasi CW-poset of rank $\leq n-\rank \sigma$.
The poset $\lk_P(\sigma)$ is called a star in \cite{EK}, but we call it a link since if $P$ is the face poset of a simplicial complex, then this poset corresponds to a link of a simplicial complex.

Let $\FF$ be a sheaf on $P$.
We write $\widetilde H_i(\mathcal F)=H_i(\C_\bullet^{\mathcal F})$.
For $\sigma \in P$,
let $\lk_{\mathcal F}(\sigma)$ (resp.\ $\cost_{\mathcal F}(\sigma)$) be the sheaf on $P$ whose stalks and restriction maps are restricted to $\lk_P(\sigma)$ (resp.\ $\cost_P(\sigma)$).
Let $M$ be a squarefree $P$-module and $\mathcal F = \mathcal F^M$.
Then, by the definition of the Karu complex, 
it is easy to see that $(\mathcal L_{\bullet-1}^M)_{\ee_\sigma} \cong \C_\bullet^{\mathcal F}/( \C_\bullet^{\cost_{\mathcal F}(\sigma)})$
for $\sigma \in P$.
Thus
\begin{align}
\label{Hformula}
H_i((\mathcal L^M_\bullet)_{\ee_\sigma}) \cong H_{i-1}(\C_\bullet^{\mathcal F}/( \C_\bullet^{\cost_{\mathcal F}(\sigma)})) \cong \widetilde H_{i+\rank \sigma-1}(\lk_{\mathcal F}(\sigma)).
\end{align}
Recall that a graded $R$-module $N$ of dimension $d$ is Cohen--Macaulay if and only if 
$\Ext_R^i(M,R(-\ichi))=0$ for $i \ne |\widehat P|-d$
(see \cite[Corollary 3.5.11]{BH}).
Then \eqref{Hformula} and Theorem \ref{3.4} imply

\begin{theorem}
\label{3.6}
Let $M$ be a squarefree $P$-module of dimension $d$ and $\mathcal F=\mathcal F^M$.
Then $M$ is Cohen--Macaulay if and only if, for any $\sigma \in P$, $\widetilde H_i(\lk_{\mathcal F}(\sigma))=0$ for all $i \ne d -1 -\rank \sigma$.
\end{theorem}

\begin{remark}
Theorems \ref{3.4} and \ref{3.6} are ring-theoretic interpretations of the results in \cite{Ka,EK}.
Indeed, Karu complexes essentially appeared in \cite[Section 2.2]{Ka}
and conditions in Theorem \ref{3.6} were used in \cite{EK} as a definition of the Cohen--Macaulay property
of sheaves.
Also, about Corollary \ref{extSqP},
the essentially same statement appeared in \cite[Lemma 5.3]{EK} for canonical modules (see next subsection).
However,
we remark that, in \cite{Ka},
the complex $\mathcal L_\bullet^M$ was treated as a complex of $K[\theta_1,\dots,\theta_d]$-module,
where $\theta_1,\dots,\theta_d$ is a certain l.s.o.p.\ of $K[P]$,
and the $R$-module structure was not given.
\end{remark}

\subsection*{Canonical modules}
For a $\ZZ^{|\widehat P|}$-graded $R$-module $M$ of dimension $d$,
the module
$$\Omega(M)=\Ext_R^{|\widehat P|-d}(M,R(-\ichi))$$
is called the {\em ($\ZZ^{|\widehat P|}$-graded) canonical module} of $M$.
By Corollary \ref{extSqP},
if $M$ is a squarefree $P$-module then so is $\Omega(M)$.
We recall some known properties of canonical modules which are used in the latter sections.

It is well-known in commutative algebra that, if $M$ is a finitely generated $\ZZ^d$-graded Cohen--Macaulay $R$-module,
then $\Omega(M)$ is Cohen--Macaulay and its $\ZZ^d$-graded Hilbert series is given by
$H_{\Omega(M)}(t_1,\dots,t_d)=(-1)^{\dim M} H_M(\frac 1 {t_1},\dots,\frac 1 {t_d})$
(see \cite[p.\ 49]{StG}).
This implies the following duality of flag $h$-vectors.

\begin{lemma}
\label{3.2}
If $M$ is a $d$-dimensional Cohen--Macaulay squarefree $P$-module,
then $\Omega(M)$ is a $d$-dimensional Cohen--Macaulay squarefree $P$-module with $\Psi_{\Omega(M)}(\aaa,\bb)=\Psi_M(\bb,\aaa)$.
\end{lemma}

Note that the latter condition in the above lemma says $h_S(\Omega(M))=h_{[d]\setminus S}(M)$ for all $S \subset [d]$.

Recall that a {\em linear system of parameters} (l.s.o.p.\ for short) of a finitely generated $\ZZ$-graded $R$-module $M$ of dimension $d$ is a sequence $\Theta=\theta_1,\dots,\theta_d\in R$ of linear forms such that $\dim_K (M/\Theta M) < \infty$,
where we consider the $\ZZ$-grading of $R$ defined by $\deg x_\sigma=1$ for all $\sigma \in \widehat P$.
An l.s.o.p.\ exists if $K$ is infinite.
See \cite[I Lemma 5.2]{StG}.
For a $\ZZ$-graded $R$-module $M$ and for an integer $k \in \ZZ$, let $M_k$ be the graded component of $M$ of degree $k$
and let $M(-k)$ be the graded module $M$ with grading shifted by $k$.
Also, we write $M^T$ for the graded Matlis dual of $M$ \cite[Section 3.6]{BH}.
The following fact is more or less well-known, but we give a proof for completeness.

\begin{lemma}
\label{3.3}
Let $M$ be a finitely generated $\ZZ$-graded Cohen--Macaulay module of dimension $d$
and $\Theta$ an l.s.o.p.\ of $M$. Then $\Theta$ is an l.s.o.p.\ of $\Omega(M)$ and
$$(M/(\Theta M))^T \cong (\Omega(M)/ (\Theta \cdot \Omega(M))) (+d).$$
\end{lemma}

\begin{proof}
If $d=0$, the assertion follows from the graded local duality 
\cite[Theorem~3.6.19(b)]{BH}. In fact, if $d=0$ then $M$ equals to its $0$th local cohomology module $H_{\frak m}^0(M)$, where ${\frak m}$  is the graded maximal ideal of $R$,
and the local duality says $H_{\frak m}^0(M)^T \cong \Omega(M)$.
Assume $d \ge 1$ and set $\Theta =\theta_1, \ldots, \theta_d$. 
By the long exact sequence of $\Ext_R^\bullet(-, R(-\ichi))$ induced  by 
$$0 \too M(-1) \stackrel{\times \theta_1} \too M \too M/\theta_1 M \too 0,$$
we have $\Omega(M/\theta_1 M) \cong (\Omega(M)/ \theta_1 \Omega(M))(+1)$. 
Repeating this argument, we have  
$$\Omega(M/\Theta M) \cong \big(\Omega(M)/ (\Theta \cdot \Omega(M))\big)(+d).$$ 
Since $M/\Theta M$ is a 0-dimensional (Cohen--Macaulay) $R$-module, we have 
$(M/\Theta M)^T \cong \Omega(M/\Theta M)$. Summing up the above equations, we get the desired statement. 
\end{proof}

\subsection*{Skeletons}
For a sheaf $\mathcal F$ on $P$, we define its dimension by $\dim \FF=\max\{ \rank \sigma: \sigma \in P, \FF_\sigma \ne 0\}$.
Thus $\dim \mathcal F=  \dim M(\mathcal F)$.
For a sheaf $\mathcal F$ on $P$ of dimension $d$,
Karu \cite{Ka} defined its dual sheaf $\mathcal F^\vee$ as follows: The stalks of $\mathcal F^\vee$ are 
defined by $\mathcal F^\vee_\sigma = H_{d-1}(\C^{\mathcal F}_\bullet / \C^{\cost_{\mathcal F}(\sigma)}_\bullet)^*$ and the restriction maps of $\mathcal F^\vee$ are the maps induced by the $K$-dual of the natural surjection 
\begin{align}
\label{4-10}
\C^{\mathcal F}_\bullet /\C^{\cost_{\mathcal F}(\tau)}_\bullet \twoheadrightarrow \C^{\mathcal F}_\bullet /\C^{\cost_{\mathcal F}(\sigma)}_\bullet
\end{align} for $\sigma >\tau$.
It is not difficult to see that
taking the dual sheaf $\FF^\vee$ is essentially the same as taking the canonical module,
namely, 
$$\mathcal F^{\Omega(M)} \cong (\mathcal F^M)^\vee.$$
Indeed, for a squarefree $P$-module $M$ of dimension $d$ with $\mathcal F =\mathcal F^M$,
one easily verifies from \eqref{Hformula} that 
$$(\mathcal F^M)^\vee_\sigma = H_{d-1}(\C_\bullet^{\mathcal F}/ \C_\bullet^{\cost_{\mathcal F}(\sigma)})^*
\cong H_d((\mathcal L_\bullet^M)_{\ee_\sigma})^* \cong \Omega(M)_{\ee_\sigma}^* \cong \mathcal F^{\Omega(M)}_\sigma$$
and that the restriction maps of $\FF^{\Omega(M)}$ are induced by the surjections \eqref{4-10}
since they correspond to the multiplication maps $\mult^\sigma_\tau: \Omega(M)_{\ee_\tau} \to \Omega(M)_{\ee_\sigma}$ 
and since, by the identifications $\Omega(M)_{\ee_\sigma} \cong H_{d-1}(\C^{\mathcal F}_\bullet /\C^{\cost_{\mathcal F}(\sigma)}_\bullet) $ and $\Omega(M)_{\ee_\tau} \cong H_{d-1}(\C^{\mathcal F}_\bullet / \C^{\cost_{\mathcal F}(\tau)}_\bullet)$,
these multiplication maps in $H_{d}(\mathcal L_\bullet^M)$ are induced from \eqref{4-10}.

For an integer $k <n$, the poset $P^{(k)}=\{\sigma \in P :\rank \sigma \leq k+1\}$ is called the {\em $k$-skeleton} of $P$.
For a sheaf $\mathcal F$ on $P$,
we define its {\em $k$-skeleton} $\mathcal F^{(k)}$ to be the sheaf whose stalks and restriction maps are restricted in $P^{(k)}$.
Also, for a squarefree $P$-module $M$,
we define its $k$-skeleton $M^{(k)}$ by
$$\textstyle M^{(k)} = M/ ( \sum_{\sigma \in P \setminus P^{(k)}} M_{\ee_\sigma} R) \cong \bigoplus_{\sigma \in P^{(k)}} M_{\ee_\sigma} \otimes_K K[\langle \sigma \rangle],$$
where the last isomorphism is an isomorphism as $K$-vector spaces.
Note that $M^{(k)}= M((\mathcal F^M)^{(k)})$
and, by the criterion of the Cohen--Macaulay property, $M^{(k)}$ is Cohen--Macaulay if $M$ is Cohen--Macaulay.
Karu proved that, for a Cohen--Macaulay sheaf $\mathcal F$ on $P$ over $\RR$ and for $k < \dim \FF -1$,
there is a surjection $(\mathcal F^{(k)})^\vee \to \mathcal F^{(k)}$
(see \cite[pp.\ 249--250]{EK}).
This result of Karu implies the following statement for canonical modules.

\begin{theorem}[Karu]
\label{KaruEmb}
Let $P$ be a quasi CW-poset and $M$ a Cohen--Macaulay squarefree $P$-module of dimension $d$ over $\RR$.
For $k<d-1$, there is an injection $M^{(k)} \to \Omega(M^{(k)})$.
\end{theorem}

\begin{proof}
Recall that $M(\mathcal F^N)=N$ for any squarefree $P$-module $N$
and that a surjection $\mathcal F \to \mathcal G$ between sheaves on $P$ induces an injection $M(\mathcal G) \to M(\mathcal F)$ by the definition of $M(-)$.
Observe $\mathcal (\mathcal F^M)^{(k)}=\mathcal F^{M^{(k)}}$.
Karu's result says that there is a surjection
from $((\mathcal F^M)^{(k)})^\vee=\mathcal (\mathcal F^{M^{(k)}})^\vee$
to $(\mathcal F^M)^{(k)} =\mathcal F^{M^{(k)}}.$
This implies that there is an injection
from $ M({\mathcal F^{M^{(k)}}})= M^{(k)}$ to 
$M((\mathcal F^{M^{(k)}})^\vee) = \Omega(M^{(k)})$ as desired.
\end{proof}

\section{The proof of Theorem~\ref{3.4}}
In this section,
we prove  Theorem~\ref{3.4} as a corollary of a more general result (Theorem~\ref{main duality}). 
Since the contents of this section is purely algebraic,
readers who are only interested in combinatorics may skip this section.
We refer the readers to \cite{Ha} for basics on the theory of derived categories.
Before proving the main result,
we discuss some properties of squarefree $P$-modules and Karu complexes.

\subsection*{Squarefree modules}
Here we recall squarefree modules over a polynomial ring introduced by the second author \cite{Ya}.
Let $A=K[x_1,\dots,x_m]$ be a polynomial ring with each $\deg x_i = \ee_i \in \ZZ^m$.
For $F \subset [m]=\{1,2,\dots,m\}$,
we write $\ee_F=\sum_{i \in F} \ee_i$ and $K[F]= A/(x_i : i \not \in F) \cong K[x_i: i \in F]$. 

\begin{definition}
\label{2.1}
A finitely generated $\NN^m$-graded $A$-module $M$ is called a {\em squarefree $A$-module}
if it satisfies that, for any $\uu=(u_1,\dots,u_m) \in \NN^m$ and for any $i \in [m]$ with $u_i>0$,
the multiplication
$$\times x_i: M_\uu \to M_{\uu+\ee_i}$$
is bijective.
\end{definition}

Let $\gMod A$ be the category of $\ZZ^m$-graded $A$-modules and their degree preserving $A$-homomorphism.
Let $\Sq A$ be the full subcategory of $\gMod A$ consisting of squarefree $A$-modules.
As shown in \cite{Ya}, $\Sq A$ is an abelian subcategory of $\gMod A$. 
Moreover, $\Sq A$ has enough injectives, and any injective object is a finite direct sum of 
copies of $K[F]$ for various $F \subset [m]$.
Below, we recall homological properties of squarefree $A$-modules studied in \cite{Ya2}.

Let ${}^* \! D_A^\bullet$ be the $\ZZ^m$-graded dualizing complex of $A$.
Thus ${}^* \! D_A^\bullet$ is a minimal injective resolution of $A(-\ichi)$, where $\ichi=\ee_{[m]} \in \ZZ^m$, in $\gMod A$ up to a translation,
and has the following description
$${}^* \! D_A^\bullet: 0 \too {}^* \! D_A^{-m} \stackrel{\partial}{\too} {}^* \! D_A^{-m+1} \stackrel{\partial}{\too} \cdots\stackrel{\partial}{\too}  {}^* \! D_A^0 \too 0$$
with
$${}^* \! D_A^{-i}= \bigoplus_{\substack{F \subset [m] \\ |F| = i}} \Est (K[F]),$$
where $\Est (K[F])$ is the injective hull of $K[F]$ in $\gMod A$. 
If we forget the grading, ${}^* \! D_A^\bullet$ is quasi-isomorphic to the usual  normalized dualizing complex.
Note that, since $^*\! D_A^\bullet$ is a $\ZZ^m$-graded injective resolution of $A(-\ichi)$, we have 
\begin{align}
\label{hoshi}
H^{-i}(\Hom_A^\bullet (M,{}^* \! D_A^\bullet)) \cong \Ext_A^{m-i}(M,A(-\ichi))
\end{align}
for any finitely generated $\ZZ^m$-graded $A$-module $M$.
As shown in \cite[Theoerem~2.6]{Ya}, if $M$ is a squarefree $A$-module, then so is $\Ext^i_A(M,A(-\ichi))$ for all $i$. 
More generally, if $M^\bullet$ is a bounded cochain complex of squarefree $A$-modules, then $H^i(\Hom_A^\bullet (M^\bullet,{}^* \! D_A^\bullet))$ is a squarefree $A$-module 
for all $i$ (see Section 3 of \cite{Ya4}).

For a $\ZZ^m$-graded $A$-module $M$, $M_{\ge \zero}$ denotes the submodule $\bigoplus_{\uu \in \NN^m} M_\uu$, and call it the 
{\it $\NN^m$-graded part} of $M$. 
Let $I_A^\bullet= ({}^* \! D_A^\bullet)_{\ge \zero}$.
Then $I_A^\bullet$ is quasi-isomorphic to ${}^* \! D_A^\bullet$ itself, and
$I_A^{-i}= \bigoplus_{F \subset [m], |F| = i} K[F]$
since $\Est (K[F])_{\geq \zero}=K[F]$ (see e.g.\ \cite[p.\ 48]{Ya2}).
Let $\Delta$ be a simplicial complex on $[m]$,
that is, a collection of subsets of $[m]$ satisfying that $F \in \Delta$ and $G \subset F$ imply $G \in \Delta$
(we assume that $\emptyset$ is an element of $\Delta$).
Consider the subcomplex of $I_A^\bullet$
$$I_\Delta^\bullet: 0 \too I_\Delta^{-m} \stackrel{\partial}{\too} I_\Delta^{-m+1} \stackrel{\partial}{\too} \cdots \stackrel{\partial}{\too} I_\Delta^0 \too 0$$
with 
$$I_\Delta^{-i}= \bigoplus_{F \in \Delta  \atop |F| = i}K[F].$$
Note that,
for $f \in K[F] \subset I_\Delta^{-i}$, one has
$$\partial(f)=\sum_{i \in F} \pm  \pi_{F,F \setminus \{ i \}}(f) \ \in 
\bigoplus_{G \in \Delta \atop |G| =i-1} K[G] = I_\Delta^{-i+1},$$
where $\pi_{F,F \setminus \{ i \}}$ is the natural subjection $K[F] \twoheadrightarrow K[F \setminus \{ i \}]$
and where $\pm$ is given by the standard incidence function of the simplicial complex $\Delta$.
We say that $M \in \Sq A$ is supported by $\Delta$ if $M_{\ee_F}=0$ for all $F \not \in \Delta$.
The following result was essentially shown in \cite{Ya4}.

\begin{lemma}
\label{pre Nn-part}
Let $\Delta$ be a simplicial complex on $[m]$.
\begin{itemize}
\item[(i)] If $M$ is a squarefree $A$-module,
then for any subset $F \subset [m]$,
$$[\Hom_A(M, K[F])]_{\ge \zero} \cong (M_{\ee_F})^* \otimes K[F].$$
\item[(ii)]
For a bounded cochain complex $M^\bullet$ of squarefree $A$-modules supported by $\Delta$,
$[\Hom_A^\bullet(M^\bullet, I_\Delta^\bullet)]_{\ge \zero}$
and $\Hom_A^\bullet(M^\bullet, {}^* \! D_A^\bullet)$ are isomorphic in the derived category.
\end{itemize}
\end{lemma}

\begin{remark}
In Lemmas \ref{pre Nn-part} and \ref{meaning of K-complex}, we consider that elements in $(M_{\ee_F})^*$ and $(M_{\ee_\sigma})^*$ have degree $\zero$.
\end{remark}

\begin{proof}[Proof of Lemma \ref{pre Nn-part}]
(i)
By \cite[Lemma~3.20]{Ya2}
we have  
$$[\Hom_A(M, \Est(K[F]))]_{\ge \zero} \cong (M_{\ee_F})^* \otimes_K K[F]$$
for any $F \subset [m]$.
However, since $\Est(K[F])_{\ge \zero} = K[F]$ and  $\Est(K[F]) \setminus K[F]$ does not  concern 
the $\NN^m$-graded part of $\Hom_A(M, \Est(K[F]))$,
we have 
\begin{align}
\label{yan}
[\Hom_A(M, K[F])]_{\ge \zero} \cong [\Hom_A(M, \Est(K[F]))]_{\ge \zero},
\end{align}
which implies the desired statement.

(ii)
Since each $M^i$ is supported by $\Delta$,
(i) says that
$[\Hom_A^\bullet(M^\bullet,I_A^\bullet)]_{\geq \zero}$
is equal to
$[\Hom_A^\bullet(M^\bullet,I_\Delta^\bullet)]_{\geq \zero}$.
Also, \eqref{yan} and the description of $^*\! D_A^\bullet$ imply
$[\Hom_A^\bullet(M^\bullet,I_A^\bullet)]_{\geq \zero} = [\Hom_A^\bullet (M^\bullet, {}^* \! D_A^\bullet)]_{\geq \zero}$.
Since homologies of $\Hom_A^\bullet (M^\bullet, {}^* \! D_A^\bullet)$ are squarefree $A$-modules by \eqref{hoshi}
and since squarefree $A$-modules are $\NN^m$-graded,
$[\Hom_A^\bullet (M^\bullet, {}^* \! D_A^\bullet)]_{\geq \zero}$ is quasi-isomorphic to $\Hom_A^\bullet (M^\bullet, {}^* \! D_A^\bullet)$,
which implies the desired statement.
\end{proof}

Consider the case when $A=R$.
Since squarefree $P$-modules are squarefree $R$-modules supported by $\Delta_P$,
Lemma \ref{pre Nn-part} implies the following fact.

\begin{corollary}
\label{4.3}
If $M$ is a squarefree $P$-module, then we have
$\Ext_R^{|\widehat P|-i}(M,R(-\ichi))
\cong H^{-i}([\Hom_R^\bullet(M, I_{\Delta_P}^\bullet)]_{\ge \zero})$
for all $i$.
\end{corollary}

Note that a squarefree $R$-module supported by $\Delta_P$ is not necessary a squarefree $P$-module
because of condition (b) of squarefree $P$-modules.

\subsection*{Category of Squarefree $P$-modules}
Here we discuss the category of squarefree $P$-modules.
Let $\Sq_P R$ be the full subcategory of $\gMod R$ consisting of squarefree $P$-modules. 
By Lemma~\ref{2.2}, $\Sq_P R$ is an abelian subcategory of $\gMod R$.
From now on, if there are no danger of confusion, $C$ and $C'$ always denote  (possibly empty) 
chains of $\widehat{P}$, equivalently, faces of $\Delta_P$.
For a chain $C$ of $\widehat{P}$, we write $K[C]=R/(x_\sigma: \sigma \not \in C).$

Recall the notion of sheaves on $P$ discussed in the previous sections. 
Let $\Sh P$ denote the category of sheaves of finite vector spaces on the poset $P$ and the maps between them.

\begin{proposition}\label{Sq=uSh}
We have the category equivalence $\Sq_P R \cong (\Sh P)^{\mathsf {op}}$, where ${\mathsf{op}}$ means the opposite category. 
\end{proposition}

\begin{proof}
This (anti)equivalence is given by the constructions  $M \mapsto \FF^M$ and $\FF \mapsto M(\FF)$ introduced in Section 2. 
For a morphism  $M \to N$ in $\Sq_P R$, let $f_\sigma: M_{\ee_\sigma} \to N_{\ee_\sigma}$ be the restriction of $f$ to the 
degree $\ee_\sigma$ part.  Then the family of $K$-linear maps 
$\{ (f_\sigma)^*\}_{\sigma \in P}$ gives a morphism $\FF^N \to \FF^M$ in $\Sh P$. 
It is easy to see that this correspondence gives a contravariant functor $\Sq_P R \to \Sh P$. 
By a similar way, we can construct a morphism $M(\GG) \to M(\FF)$ in $\Sq_P R$ from a morphism 
$\FF \to \GG$ in $\Sh P$, which gives a contravariant functor $\Sh P \to \Sq_P R$.  
Then, since $M(\FF^M) \cong M$ and $\FF^{M(\FF)} \cong \FF$, we have $\Sq_P R \cong (\Sh P)^{\mathsf {op}}$. 
\end{proof}

\begin{remark}\label{convention of sheaf}
In \cite{Ya3}, a sheaf on a finite poset is defined in the opposite manner. 
More precisely, the restriction maps of a sheaf $\FF$ in \cite{Ya3} are $K$-linear maps $\FF_\sigma \to \FF_\tau$ for $\sigma < \tau$. 
If we use this convention, the category of sheaves on $P$ is directly equivalent to $\Sq_P R$, and we do not have to take 
the opposite category. However, we follow the convention of \cite{Ka, EK}  in this paper.  
\end{remark}

\begin{corollary}
The category $\Sq_P R$ has  enough injectives,  and any injective object is a finite direct sum of 
copies of $K[\langle \sigma \rangle]$  for various $\sigma \in P$. 
\end{corollary}

\begin{proof}
%
It is well-known that $\Sh P$ is an abelian category with enough projectives and injectives, and 
an indecomposable projective object is of the form  $\mathcal P(\sigma)$ for some $\sigma \in P$, 
where 
$$\mathcal P(\sigma)_\tau=\begin{cases}
K, & \text{if $\tau \le \sigma$,}\\
0, & \text{otherwise,}
\end{cases}$$
and all the restriction map $\res_\rho^\tau : \mathcal P(\sigma)_\tau \to  \mathcal P(\sigma)_\rho$ are injective for all 
$\tau, \rho \in P$ with $\tau > \rho$. 
See, for example, \cite{Ya3} and references cited therein (the reader should be careful with the point that the convention 
on sheaves in \cite{Ya3} is ``opposite" to ours as mentioned in Remark~\ref{convention of sheaf}).

Since $\Sq_P R \cong (\Sh P)^{\mathsf {op}}$, $\Sq_P R$ also has enough projectives and injectives, 
and injective objects in $\Sq_P R$ correspond to projective objects in $\Sh P$.
Moreover, it is easy to check that  $M(\mathcal P(\sigma)) \cong K[\langle \sigma \rangle]$. So we are done.
\end{proof}

\subsection*{Another description of a Karu complex}
Next, we show that the Karu complex $\mathcal L^M_\bullet$ can be described in a way similar to Lemma \ref{pre Nn-part}.
We define the complex $J_P^\bullet$ by $J_P^\bullet=\mathcal L_{-\bullet}^{K[P]}$.
Thus, $J_P^\bullet$ is the complex of squarefree $P$-modules of the form 
$$J_P^\bullet: 0 \too J_P^{-n} \stackrel{\widetilde \partial}{\too} J_P^{-n+1} \stackrel{ \widetilde \partial}{\too} \cdots \stackrel{\widetilde \partial}{\too}  J_P^0 \too 0$$
with $J_P^{-i}= \bigoplus_{\substack{\sigma \in P_i}}K[\langle \sigma \rangle]$.
Recall that, for $f \in K[\langle \sigma \rangle]$ with $\rank \sigma=i$,
$$ \widetilde \partial(f)= \sum_{\tau \in P \atop \tau <\sigma} \ve(\sigma, \tau) \cdot \pi_{\sigma, \tau}(f) \ \in 
J_P^{-i+1}=\bigoplus_{\tau \in P_{i-1}  } K[\langle \tau \rangle ],$$
where $\ve$ is the fixed incidence function and $\pi_{\sigma, \tau}$ is the natural subjection $K[\langle \sigma \rangle ] \twoheadrightarrow K[\langle \tau \rangle]$.

\begin{lemma}\label{meaning of K-complex}
Let $M \in \Sq_P R$ and $\sigma \in P$ with $\rank \sigma=r$.
\begin{itemize}
\item[(i)]
$[\Hom_R(M, K[\langle \sigma \rangle])]_{\ge \zero} \cong (M_{\ee_\sigma})^* \otimes_K K[\langle \sigma \rangle].$
\item[(ii)] $[\Hom_R^\bullet(M,J_P^\bullet )]_{\geq \zero} \cong \mathcal L_{-\bullet}^M$.
\item[(iii)] $H_k(\mathcal L_\bullet^{K[\langle \sigma \rangle]}) = 0$ for $k\ne r$ and 
$H_r(\mathcal L_\bullet^{K[\langle \sigma \rangle]}) = x_\sigma K[\langle \sigma \rangle]$.
\end{itemize}
\end{lemma}

\begin{proof}
(i)
We may assume that  $\sigma \ne \hat 0$. 
Since $\partial \sigma$ is Gorenstein*, $\Omega(K[\partial \sigma])=K[\partial \sigma]$.
Then, since $K[\langle \sigma \rangle]=K[\partial \sigma ][x_\sigma]$,
the canonical module $\Omega (K[\langle \sigma \rangle ]) \cong K[\langle \sigma \rangle ](-\ee_\sigma)$ is isomorphic to $x_\sigma K[\langle \sigma \rangle]$ the ideal of $K[\langle \sigma \rangle]$ generated by $x_\sigma$.
Let
$$0 \too K[\langle \sigma \rangle] \too  I^0 \stackrel f \too I^1$$
be the first step of the  minimal injective resolution of $K[\langle \sigma \rangle ]$ in the category $\Sq R$. 
By \cite[Proposition~3.5]{Ya4}, 
each $I^i$ is the direct sum of $\dim_K \Omega(K[\langle \sigma \rangle])_{\ee_F}$ copies of $R/(x_\sigma  :  \sigma \not \in F)$ for $F \subset \widehat P$ with $|F|=r-i$.
Then, since $\Omega(K[\langle \sigma \rangle]) \cong x_\sigma K[\langle \sigma \rangle]$, 
$$I^0=  \bigoplus_{\substack{\max C = \sigma \\ |C| = r }} K[C] \qquad \mbox{ and } \qquad 
I^1=\bigoplus_{\substack{\max C' = \sigma \\ | C'| = r-1 }} K[{C'}].$$
Observe $$\Hom_R(M,K[C])_{\geq \zero} \cong (M_{\sum_{\sigma \in C} \ee_\sigma})^* \otimes K[C] \cong 
(M_{\ee_{\max C}})^* \otimes K[C]$$
by Lemma \ref{pre Nn-part} and
condition (b) of squarefree $P$-module.
Since $\Hom_R(M,K[\langle \sigma \rangle])$ is the kernel of
$$f_*: \Hom_R \left( M,   I^0 \right) \too \Hom_R \left(M,  I^1 \right),$$ 
$[\Hom_R(M, K[\langle \sigma \rangle ])]_{\ge \zero}$
is isomorphic to the kernel of
$$(M_\sigma)^* \otimes_K f : 
\bigoplus_{\substack{\max C = \sigma \\ |C| = r }}  (M_{\ee_\sigma})^* \otimes_K K[C] \too 
\bigoplus_{\substack{\max C' = \sigma \\ |C'| = r-1}}  (M_{\ee_\sigma})^* \otimes_K K[C']. 
$$ 
Then the statement follows since $\ker ((M_{\ee_\sigma})^* \otimes_K f) \cong (M_{\ee_\sigma})^* \otimes_K \ker f$
and since $\ker f = K[\langle \sigma \rangle]$.

(ii) is an immediate consequence of (i).
We prove (iii).
Recall that by \eqref{Hformula} in Section 3, one has
$$H_k((\mathcal L_\bullet^{K[\langle \sigma \rangle]})_{\ee_\tau}) \cong \widetilde H_{k-1-\rank \tau} (\lk_{\langle \sigma \rangle}(\tau))$$
for all $\tau \leq \sigma$.
Since $\lk_{\langle \sigma \rangle}(\tau)$ has the maximal element,
its order complex is a cone unless $\sigma = \tau$.
Thus the homologies of $\lk_{\langle \sigma \rangle}(\tau)$ are zero except for the case when $k=r$ and $\tau = \sigma$,
which implies the first assertion.
Also, since $\mathcal L_r^{K[\langle \sigma \rangle]}=K[\langle \sigma \rangle]$,
$H_r(\mathcal L_\bullet^{K[\langle \sigma \rangle]})$ is an ideal of $K[\langle \sigma \rangle]$.
However, if an ideal in $K[\langle \sigma \rangle]$ is a squarefree $P$-module then it must be generated by variables.
Thus the above fact on homologies of $\lk_{\langle \sigma \rangle}(\tau)$ implies that
$H_r(\mathcal L_\bullet^{K[\langle \sigma \rangle]})$ is the ideal generated by $x_\sigma$.
\end{proof}

We also note the following fact which will be obvious to the specialists.

\begin{lemma}
\label{functor}
$[\Hom_R^\bullet(-,I_{\Delta_P}^\bullet) ]_{\geq \zero}$ and
$[\Hom_R^\bullet(-, J_P^\bullet)]_{\ge \zero}$ give contravariant functors from the bounded derived category $\Db(\Sq_P R)$ to $\Db(\gMod R)$.   
\end{lemma}

\begin{proof}
Since $\Hom_R^\bullet(-,{}^*\! D^\bullet_R)$ is a contravariant functor in $\Db(\gMod R)$,
the statement for $I_{\Delta_P}^\bullet$ follows from Lemma \ref{pre Nn-part}(ii).
We consider $J_P^\bullet$.
Let $M^\bullet$ be a bounded cochain complex of squarefree $P$-modules which is acyclic (i.e., $H^i(M^\bullet)=0$ for all $i$).
What we must prove is that
$[\Hom_R^\bullet(M^\bullet, J_P^\bullet)]_{\geq \zero}$ is also acyclic.
By Lemma \ref{meaning of K-complex}(i),
$[\Hom_R^\bullet(M^\bullet, K[\langle \sigma \rangle])]_{\geq \zero}$ is acyclic for all $\sigma \in P$.
Recall that each $J_P^i$ is a finite direct sum of copies of  $K[\langle \sigma \rangle]$.  
Hence, by the usual double complex argument, we can show that $[\Hom_R^\bullet(M^\bullet, J_P^\bullet)]_{\geq \zero}$ is acyclic. 
\end{proof}

\subsection*{Main result.}
We will construct the chain map $$\tilde{\iota}: J_P^\bullet \to I_{\Delta_P}^\bullet$$  
and prove that this chain map induces a quasi-isomorphism from
$[\Hom_R^\bullet(M^\bullet ,J_P^\bullet)]_{\geq \zero}$
to $[\Hom_R^\bullet(M^\bullet ,I_{\Delta_P}^\bullet)]_{\geq \zero}$ for any bounded cochain complex $M^\bullet$ of squarefree $P$-modules.

\begin{construction}
We construct the map $\tilde{\iota}$.
We write $\partial^{-i}: I_{\Delta_P}^{-i} \to I_{\Delta_P}^{-i+1}$ for the boundaries of $I_{\Delta_P}^\bullet$.
Fix an incidence function $\varepsilon$ of $P$.
Take $\sigma \in P$ with $\rank \sigma=r$. 
The complex $I_\sigma^\bullet = [\Hom_R^\bullet(K[\langle \sigma \rangle], I_{\Delta_P}^\bullet)]_{\geq \zero}$ can be seen as a subcomplex of $I_{\Delta_P}^\bullet$
with $I_\sigma^{-i}= \bigoplus_{\max C\leq \sigma, |C|=i} K[C]$ 
by Lemma \ref{pre Nn-part}(i).
The ``tail"  $0 \to I_\sigma^{-r} \to I_\sigma^{-r+1}$  of  the complex $I_\sigma^\bullet$ is of the form 
$$0 \too \bigoplus_{\substack{\max C = \sigma \\  |C| = r }} K[C] \stackrel{\partial^{-r}}{\too}\bigoplus_{\substack{\max C' \le \sigma \\ |C'| =r-1}} K[C'].$$
Since $\ker (\partial^{-r})= \Omega(K[\langle \sigma \rangle])$ by Corollary \ref{4.3},
as we saw in the proof of Lemma \ref{meaning of K-complex}(i),
$\ker(\partial^{-r})$ is isomorphic to $x_\sigma K[\langle \sigma \rangle]$.
Then, by the injectivity of $K[C]$'s  in $\Sq R$,  we have an injection 
$$\iota_\sigma : K[\langle \sigma \rangle] \too \bigoplus_{\substack{\max C = \sigma \\ |C| = r }} K[C]$$
satisfying $\partial^{-r} \circ \iota_\sigma(x_\sigma)=0$
by extending the injection 
$$x_\sigma K[\langle \sigma \rangle] \cong \ker(\partial^{-r}) \hookrightarrow \bigoplus_{\substack{\max C = \sigma \\ |C| = r }} K[C].$$ 
Note that $\iota_\sigma$ is unique up to constant multiplications. 
More precisely, if an injective homomorphism 
$\iota' :  K[\langle \sigma \rangle] \too I_\sigma^{-r}$ 
satisfies $\partial^{-r} \circ \iota'(x_\sigma)=0$, then we have $\iota'= a \cdot \iota_\sigma$ for some $a \in K\setminus \{0\}$. 
This is because 
$\iota'$ only depends on the choice of $\iota'(1)$ but,
by the injectivity of the multiplication by $x_\sigma$ in $I_\sigma^{-r}=\bigoplus_{\max C = \sigma,\ |C| = r } K[C]$,
it actually only depends on the choice of $\iota'(x_\sigma) \in \ker (\partial^{-r})_{\ee_\sigma} \cong K$.
Also, by the injectivity of $\iota_\sigma$, one has
\begin{align}
\label{yyy1}
\Image(\partial^{-r} \circ \iota_\sigma) \cong K[\langle \sigma \rangle]/(x_\sigma K[\langle \sigma \rangle]).
\end{align}

We claim that, by appropriate choices of $\iota_\sigma$, the maps $\{\iota_\sigma\}_{\sigma \in P}$
induces a chain map from $J_P^\bullet$ to $I_{\Delta_P}^\bullet$.

Fix $\{\iota_\sigma\}_{\sigma \in P}$.
Let $L_\sigma$ be the $1$-dimensional $K$-vector space spanned by $\iota_\sigma(1)$
and $L^{-i}=\bigoplus_{\sigma \in P_i} L_\sigma$.
We first prove that $(L^\bullet,\partial)$ is a complex.
Since $x_\sigma \partial^{-r}(\iota_\sigma(1))=\partial^{-r}(\iota_\sigma(x_\sigma))=0$, where $r=\rank \sigma$,
and since the multiplication by $x_\sigma$ in $K[C]$ with $\sigma \in C$ is injective,
$$\partial^{-r} (\iota_\sigma(1)) \in \bigoplus_{\substack{\max C<\sigma\\ |C|=r-1}} K[C] = \bigoplus_{\sigma \mbox{ \tiny covers }\tau } I_\tau^{-r+1}.$$
Let $\partial^{-r}(\iota_\sigma(1))= \sum f_\tau$ with $f_\tau \in I_\tau^{-r+1}$.
Since $(I_\rho^{-r+1})_{\ee_\tau}=0$ for all $\tau,\rho \in P_{r-1}$ with $\tau \ne \rho$,
$\partial^{-r}(\iota_\sigma(x_\tau))=x_\tau f_\tau$.
Since $\partial^{-r+1} \circ \partial ^{-r}=0$,
$x_\tau f_\tau$ is contained in the kernel of the map $\partial^{-r+1}: I_\tau^{-r+1} \to I_\tau^{-r+2}$.
Moreover, since $\partial^{-r}(\iota_\sigma(x_\tau)) \ne 0$ by \eqref{yyy1},
$x_\tau f_\tau \ne 0$.
Then, by the injectivity of the multiplication by $x_\tau$ in $I_\tau^{-r+1}$ and by the uniqueness of the map $\iota_\tau$,
there is an $a_{\sigma,\tau} \in K \setminus \{0\}$ such that $f_\tau=a_{\sigma,\tau} \cdot \iota_\tau(1)$.
This implies that
\begin{align}
\label{yyy2}
\partial^{-r}(\iota_\sigma(1)) = \sum_{\sigma \mbox{ \tiny covers } \tau} a_{\sigma,\tau} \cdot \iota_\tau(1) \in L^{-r+1},
\end{align}
and therefore $L^\bullet$ is a complex.

Observe that $L^\bullet$ is a complex of $K$-vector spaces with basis $\{\iota_\sigma(1)\}_{\sigma \in P}$
and that, for every $\sigma >\rho$ in $P$ with $\rank \sigma =\rank \rho+2$,
the set $\{\tau \in P: \sigma > \tau >\rho\}$ contains exactly two elements.
Since all the coefficients $a_{\sigma,\tau}$ are non-zero in \eqref{yyy2},
the numbers $\{a_{\sigma,\tau}\}$ give an incidence function on $P$ (in other words, $L^\bullet$ is the augmented oriented chain complex of $P$).
Thus, by the uniqueness of an incidence function of $P$,
by replacing $\iota_\sigma$ with its scalar multiple if necessary 
we may assume that the equation 
\begin{equation}\label{commute with incidence}
 \partial^{-r} \circ \iota_\sigma(1) = \sum_{\sigma \mbox{ \tiny covers }\tau} \varepsilon(\sigma,\tau) \cdot \iota_\tau(1)
\end{equation}
holds for all $\sigma \in P$.
For each $i$ with $0 \le i \le r$, define the map 
$\tilde{\iota} : J_P^\bullet \too  I_{\Delta_P}^\bullet$
by 
$\tilde{\iota}^{-i}= \sum_{\sigma \in P_i} {\iota}_\sigma.$
Clearly, \eqref{commute with incidence} says that
$\tilde \iota$ is a chain map.
\end{construction}

Now we are in the position to prove the main result of this section.

\begin{theorem}\label{main duality}
For a bounded cochain complex $M^\bullet$ of squarefree $P$-modules,
the complexes $[\Hom_R^\bullet(M^\bullet,J^\bullet_P)]_{\geq \zero}$ and $[\Hom_R^\bullet(M^\bullet, I_{\Delta_P}^\bullet )]_{\geq \zero}$
are isomorphic in the bounded derived category $\Db(\gMod R)$.
\end{theorem}

\begin{proof}
By Lemma \ref{functor}, $[\Hom_R^\bullet(-, J_P^\bullet)]_{\ge \zero}$ and $[\Hom_R^\bullet(-, I_{\Delta_P}^\bullet)]_{\ge \zero}$
are contravariant functors from $\Db(\Sq_P R)$ to  $\Db(\gMod R)$. 
Consider the chain map $\tilde{\iota}: J_P^\bullet \too I_{\Delta_P}^\bullet$. 
Taking the $\NN^{|\widehat P|}$-graded part of $\tilde{\iota}_*: \Hom_R(M^\bullet, J_P^\bullet) \to 
\Hom_R^\bullet(M^\bullet, I_{\Delta_P}^\bullet)$, we have the chain map 
$$[\Hom_R^\bullet(M^\bullet, J_P^\bullet)]_{\ge \zero} \too 
[\Hom_R^\bullet(M^\bullet, I_P^\bullet)]_{\ge \zero}.$$ 
This gives a natural transform $\eta: [\Hom_R^\bullet(-, J_P^\bullet)]_{\ge \zero} \to [\Hom_R^\bullet(-,  I_{\Delta_P}^\bullet)]_{\geq \zero}$. 
By the construction of the chain map $\tilde{\iota}$, it follows that $\eta(K[\langle \sigma \rangle])$ is quasi-isomorphism for all $\sigma \in P$. 
In fact, if $\rank \sigma =r$, then, since $K[\langle \sigma \rangle]$ is Cohen--Macaulay,
both $[\Hom_R^\bullet(K[\langle \sigma \rangle], J_P^\bullet)]_{\ge \zero}$ 
and $[\Hom_R^\bullet(K[\langle \sigma \rangle], I_{\Delta_P}^\bullet)]_{\ge \zero}$ are exact except 
at the $(-r)$-th cohomology, which is isomorphic to the ideal $x_\sigma K[\langle \sigma \rangle]$
by Corollary \ref{y4.3} and Lemma \ref{meaning of K-complex}.
Also, the map between $(-r)$-th term of the complexes coincides with the map $\iota_\sigma$
that sends $x_\sigma K[\langle \sigma \rangle]$ to the kernel of $\partial: I_\sigma^{-r} \to I_{\sigma}^{-r+1}$
which is isomorphic to $x_\sigma K[\langle \sigma \rangle]$.
It means that $\eta(-)$ is quasi-isomorphism for any injective object in $\Sq_P R$. 
Hence applying \cite[Proposition 7.1]{Ha}, we see that $\eta$ is a natural isomorphism. 
\end{proof}

By Corollary \ref{y4.3} and Lemma \ref{meaning of K-complex}(ii),
Theorem \ref{3.4} is the special case of Theorem \ref{main duality} when $M^\bullet$ is a single module.
Indeed, for a squarefree $P$-module $M$, we have an isomorphisms
$$\Ext_R^{|\widehat P|-i}(\!M,\! R(-\ichi))\! \cong\! H^{-i}([\Hom_R^\bullet(\!M,\! I_{\Delta_P}^\bullet )]_{\geq \zero})
\!\cong\! H^{-i}([\Hom_R^\bullet(\!M,\! J_P^\bullet )]_{\geq \zero})\! \cong\! H_i(\mathcal L_\bullet^M).$$

\begin{corollary}\label{qism}
The chain map $\tilde{\iota}: J_P^\bullet \to I_{\Delta_P}^\bullet$ defined above is a quasi-isomorphism. 
\end{corollary}

\begin{proof}
We use the notation in the proof of Theorem~\ref{main duality}. 
Since $[\Hom_R^\bullet(K[P], J_P^\bullet)]_{\ge \zero}\\= J_P^\bullet$, $[\Hom_R^\bullet(K[P],  I_{\Delta_P}^\bullet)]_{\ge \zero}=I_{\Delta_P}^\bullet$ and $\eta(K[P])=\tilde{\iota}$, 
the assertion follows from the fact that $\eta$ is a natural isomorphism. 
\end{proof}

\section{Extended $\cc\dd$-indices}
In this section, we consider an extension of the $\cc\dd$-index to quasi CW-posets.
Recall that $\ZZ\langle \cc,\dd\rangle$ denotes the non-commutative polynomial ring with coefficients in $\ZZ$ with the variables
$\cc=\aaa+\bb$ and $\dd = \aaa \bb + \bb \aaa$.

We say that a squarefree $P$-module $M$ of dimension $d$ has the {\em symmetric flag $h$-vector}
if $h_S(M)=h_{[d]\setminus S}(M)$ for all $S \subset [d]$,
equivalently, $\Phi_M(\aaa,\bb)=\Phi_M(\bb,\aaa)$.
Note that, if $P$ is Gorenstein*, then $K[P]$ has the symmetric flag $h$-vector
(see \cite[Corollary 3.16.6]{St12}).

\begin{lemma}
\label{4.1}
If $M$ is a squarefree $P$-module of dimension $d$,
then there are unique $\cc\dd$-polynomials $\Phi, \Upsilon \in \ZZ\langle \cc, \dd\rangle$ of degrees $d$ and $d-1$  such that
\begin{align}
\label{4-1}
\Psi_M(\aaa,\bb) &= \Phi+ \Upsilon \bb
\end{align}
Moreover, if $M$ has the symmetric flag $h$-vector then $\Psi_M(\aaa,\bb)=\Phi$.
\end{lemma}

\begin{proof}
We first prove the existence.
By Lemma \ref{2.6},
we have
$$\Psi_M(\aaa,\bb) = (\dim_K M_\zero)\cdot (\aaa -\bb)^d
+ \sum_{ \sigma \in \widehat P}(\dim_K M_{\ee_\sigma}) \Psi_{\partial \sigma}(\aaa,\bb)\cdot\bb (\aaa-\bb)^{d-\rank \sigma}.$$
Observe $(\aaa -\bb)^d=(\cc-2\bb)(\aaa -\bb)^{d-1}$ and each $\Psi_{\partial \sigma}(\aaa,\bb) \in \ZZ\langle \cc,\dd\rangle$.
Then, to prove the statement,
it is enough to prove that, for any $\Phi,\Upsilon \in \ZZ\langle \cc,\dd \rangle$,
there are $\Phi',\Upsilon' \in \ZZ\langle \cc,\dd \rangle$ such that
$(\Phi+\Upsilon \bb)(\aaa-\bb)= \Phi' + \Upsilon' \bb$.
Indeed the following computation proves the desired statement.
$$
(\Phi+\Upsilon \bb)(\aaa-\bb)=\Phi\cdot(\aaa -\bb) + \Upsilon\cdot(\bb\aaa -\bb^2) = \Phi\cdot(\cc - 2 \bb)+\Upsilon\cdot(\dd-\cc\bb).$$

Next, the uniqueness of the expression \eqref{4-1} follows since if $\Psi_M(\aaa,\bb)=\Phi+\Upsilon\bb$,
then we have $\Psi_M(\aaa,\bb) -\Psi_M(\bb,\aaa)=\Upsilon(\bb-\aaa)$,
which says that $\Psi_M(\aaa,\bb)$ determines $\Upsilon$.
Finally, if $M$ has the symmetric flag $h$-vector and $\Psi_M(\aaa,\bb)=\Phi+\Upsilon\bb$, then one obtains $\Phi+\Upsilon\bb= \Phi+\Upsilon\aaa$,
which implies $\Upsilon=0$.
\end{proof}

We call \eqref{4-1} the {\em $\bb$-expression} of $\Psi_M(\aaa,\bb)$.
By substituting $\bb=\cc-\aaa$ to \eqref{4-1}, one obtains a similar expression
$$\Psi_M(\aaa,\bb)=\Phi'+\Upsilon'\aaa,$$
where $\Phi'=\Phi + \Upsilon \cc$ and $\Upsilon'=-\Upsilon$.
We call the above expression the {\em $\aaa$-expression} of $\Psi_M(\aaa,\bb)$.
These expressions are not always non-negative.
We discuss their non-negativity later in Corollary \ref{4.5}.

Recall that $F_n$ denotes the $n$th Fibonacci number defined by $F_1=F_2=1$ and $F_{k+2}=F_{k+1}+F_k$.
Since $\cc\dd$-polynomials of degree $n$ have at most $F_{n+1}$ coefficients,
the $\aaa$-expression gives a way to express flag $h$-vectors of CW-posets of rank $n$ by
$F_{n+1}+F_n=F_{n+2}$ integers.
We prove that the $\aaa$-expression gives an efficient way to express the flag $h$-vectors
in the sense that it incorporates all linear equations satisfied by the flag $h$-vectors of all quasi CW-posets.
Let $\mathcal H_n \subset \ZZ^{2^n}$ be the set of all flag $h$-vectors of quasi CW-posets of rank $n$
and $\mathcal {HP}_n \subset \mathcal H_n$ the set of all flag $h$-vectors of the face posets of the polyhedral complexes of dimension $n-1$.
Let $\RR\mathcal H_n$ (resp.\ $\RR \mathcal {HP}_n$) be the $\RR$-linear space spanned by $\mathcal H_n$ (resp.\ $\mathcal {HP}_n$).
The next result shows that the existence of the $\aaa$-expression describes all linear equations satisfied by the flag $h$-vectors of all quasi CW-posets (or all polyhedral complexes).

\begin{proposition}
\label{4.2}
$\dim_\RR \RR \mathcal H_n=\dim_\RR \RR \mathcal {HP}_n = F_{n+2}$.
\end{proposition}

\begin{proof}
Since the number of $\cc\dd$-monomials of degree $n$ is $F_{n+1}$,
Lemma \ref{4.1} says that $\dim \RR \mathcal H_n \leq F_{n+2}$.
Thus, to prove the statement, it is enough to find $F_{n+2}$ polyhedral complexes of dimension $n-1$ whose $\aaa\bb$-indices are linearly independent.
Note that to prove the linear independence of the $\aaa\bb$-indices it is enough to prove the linear independence of their $\aaa$-expressions.

For a convex polytope $Q$, we write $\Psi_Q$ and $\Psi_{\partial Q}$ for the $\aaa\bb$-indices of the face posets of $Q$ and $\partial Q$ respectively.
Note that $\Psi_Q= \Psi_{\partial Q} \cdot \aaa$ since $Q$ and $\partial Q$ have the same flag $h$-vectors.
By \cite[Proposition 2.2]{BB}, there are $n$-polytopes $Q_1,\dots,Q_{F_{n+1}}$ and $(n-1)$-polytopes $Q'_1,\dots,Q'_{F_n}$ such that $\Psi_{\partial Q_1},\dots,\Psi_{\partial Q_{F_{n+1}}} \in \ZZ \langle \cc,\dd\rangle$ are linearly independent polynomials of degree $n$ and $\Psi_{\partial Q'_1},\dots,\Psi_{\partial Q'_{F_{n}}} \in \ZZ \langle \cc,\dd\rangle$ are linearly independent  polynomials of degree $n-1$.
Then, since the $\aaa$-expressions of the polynomials
$$\Psi_{\partial Q_1},\dots,\Psi_{\partial Q_{F_{n+1}}},\Psi_{Q'_1}=\Psi_{\partial Q'_1}\cdot \aaa,\dots,\Psi_{Q'_{F_n}}=\Psi_{\partial Q'_{F_n}}\cdot \aaa$$
are linearly independent, we obtain the desired statement.
\end{proof}

It is possible to find linear equations that determine $\RR \mathcal H_n$ in the same way as in the proof of \cite[Theorem 2.1]{BB}.
This will give another proof of Proposition \ref{4.2}.

Next, we prove Theorem \ref{main1}.
Before proving it,
we translate Karu's proof of the non-negativity of the $\cc\dd$-indices of Gorenstein* posets in the language of commutative algebra.
For a Cohen--Macaulay squarefree $P$-module $M$
such that there is an injection $\phi: M \to \Omega(M)$,
we write $\Omega(M)/M=\Omega(M)/\phi(M)$ to simplify the notation.
The following statement is due to Karu \cite[Lemma 4.7]{Ka}.

\begin{lemma}[Karu]
\label{y4.3}
Let $M$ be a Cohen--Macaulay squarefree $P$-module of dimension $d$ such that there is an injection $M \to \Omega(M)$
and $\Psi_M(\aaa,\bb)=\Phi+ \Upsilon \bb$ the $\bb$-expression of $\Psi_M(\aaa,\bb)$.
Then $\Omega(M)/M$ is a Cohen--Macaulay squarefree $P$-module with $\Psi_{\Omega(M)/M}(\aaa,\bb)=\Upsilon$.
\end{lemma}

\begin{proof}
The Cohen--Macaulay property is standard in commutative algebra.
Indeed, since $\Omega(M)$ and $M$ have the same dimension and the multiplicity (see \cite[Proposition 4.1.9 and Corollary 4.4.6(a)]{BH}),
$\Omega(M)/M$ has dimension at most $d-1$ since the multiplicity is the leading coefficient of the Hilbert polynomial.
Then, the short exact sequence $0 \to M \to \Omega(M) \to \Omega(M)/M \to 0$ and the depth lemma \cite[Proposition 1.2.9]{BH} prove that $\Omega(M)/M$ is either zero or a Cohen--Macaulay module of dimension $d-1$.

It remains to compute the $\aaa\bb$-index of $\Omega(M)/M$.
For an $\aaa\bb$-polynomial $\Psi= \sum_{S \subset [d]} a_S  w_S \in \ZZ \langle \aaa,\bb \rangle$,
we write $\pi (\Psi)= \sum_{S \subset [d]} a_S \ttt^S$.
Since
$$\Psi_{\Omega(M)}(\aaa,\bb)-\Psi_M(\aaa,\bb)= \Psi_M(\bb,\aaa)-\Psi_M(\aaa,\bb)=\Upsilon \cdot(\aaa-\bb),$$
we have
\begin{align*}
H_{\Omega(M)/M}(t_1,\dots,t_d) & = \frac  {\pi(\Upsilon \cdot (\aaa - \bb) )} {(1-t_1) \cdots (1-t_{d-1})(1-t_d)}= \frac {\pi(\Upsilon)} {(1-t_1) \cdots (1-t_{d-1})} 
\end{align*}
which shows that $\Psi_{\Omega(M)/M}(\aaa,\bb)=\Upsilon$.
\end{proof}

Let $M$ be a squarefree $P$-module of dimension $d$.
By considering the $\bb$-expression, we have the decomposition
\begin{align}
\label{4-2}
\Psi_M(\aaa,\bb)= \Phi_{-1}\cdot \cc^d + \Phi_0 \cdot \dd\cc^{d-2} + \Phi_1 \cdot \dd\cc^{d-3} + \cdots + \Phi_{d-2} \cdot \dd + \Upsilon \cdot \bb,
\end{align}
where $\Phi_{-1} \in \ZZ$, $\Phi_i$ is a $\cc\dd$-polynomial of degree $i$ for $i=0,1,\dots,d-2$ and $\Upsilon$ is the $\cc\dd$-polynomial of degree $d-1$ which appears in the $\bb$-expression.
Then, for $k<d-1$, since the flag $h$-vector of the $k$-skeleton $M^{(k)}$ is given by $(h_S(M): S \subset [k+1])$,
we have
\begin{align*}
\Psi_{M^{(k)}}(\aaa,\bb)= \Phi_{-1}\cdot \cc^{k+1} + \Phi_0 \cdot \dd\cc^{k-1} + \cdots + \Phi_{k-1} \cdot \dd + \Phi_k \cdot \bb.
\end{align*}
Observe that if $M$ is Cohen--Macaulay then so is
$M^{(k)}$ by Theorem \ref{3.6}.
The above fact and Theorem \ref{KaruEmb} show

\begin{corollary}
\label{4.4}
With the same notation as above, if $M$ is a Cohen--Macaulay squarefree $P$-module over $\RR$,
then $\Psi_{\Omega(M^{(k)})/M^{(k)}}(\aaa,\bb)=\Phi_k$.
\end{corollary}

By using Corollary \ref{4.4}, Karu \cite{Ka} proved the non-negativity of the $\cc\dd$-indices of Gorenstein* posets.
Moreover, Ehrenborg and Karu \cite[Theorem 5.6]{EK} proved that if $M$ is a Cohen--Macaulay squarefree $P$-module and if the $\aaa\bb$-index of $M$
can be written in the form $\Phi_M(\aaa,\bb)=\Phi + \Upsilon \aaa$, then $\Phi$ is non-negative.
(The statements are written in the language of sheaves.)
Since $\aaa$-expression always exists by Lemma \ref{4.1}, we have the following result.

\begin{corollary}
\label{4.5}
Let $M$ be a Cohen--Macaulay squarefree $P$-module over $\RR$ of dimension $d$,
and let
$$\Psi_M(\aaa,\bb)=\Phi+ \Upsilon  \aaa =\Phi'+ \Upsilon' \bb$$ 
be the $\aaa$-expression and the $\bb$-expression of $\Psi_M(\aaa,\bb)$ respectively.
\begin{itemize}
\item[(i)] The coefficients of $\Phi (\cc,\dd)$ and $\Phi'(\cc,\dd)$ are non-negative.
\item[(ii)] If there is an injection $M \to \Omega(M)$ then the coefficients of $\Upsilon'(\cc,\dd)$ are non-negative.
\item[(iii)] If there is an injection $\Omega(M) \to M$ then the coefficients of $\Upsilon(\cc,\dd)$ are non-negative
\end{itemize}
\end{corollary}

\begin{proof}
We first prove the non-negativity of $\Phi'$ by induction on $d$.
The statement is obvious when $d=0$.
If $d=1$, then the desired statement follows since
$$\Psi_M(\aaa,\bb)=h_\emptyset(M) \aaa + h_{\{1\}}(M)\bb=h_\emptyset(M) \cc + (h_{\{1\}}(M)-h_\emptyset(M)) \bb$$
and $h_\emptyset(M) =f_\emptyset(M) \geq 0$.
For $d>1$, the non-negativity of $\Phi'$ follows from Corollary \ref{4.4} and the induction hypothesis.

The non-negativity of $\Phi$ follows from the duality of the $\aaa\bb$-index in Lemma \ref{3.2} since it says $\Phi_{\Omega(M)}(\aaa,\bb)=\Phi_M(\bb,\aaa)=\Phi+ \Upsilon \bb$.
Also, (ii) follows from (i) and Lemma \ref{y4.3}, and,
since $\Omega(\Omega(M))=M$, (iii) follows from (ii) and the duality of the $\aaa\bb$-index.
\end{proof}

\begin{remark}
\label{rem4.6}
The coefficients of $\Phi'$ can be computed by embedding skeletons into its canonical module repeatedly.
For example, if $N$ is a Cohen--Macaulay squarefree $P$-module of dimension $8$ with the $\bb$-expression $\Psi_M= \Phi + \Upsilon \bb$
and if $\gamma$ is the coefficient of $\cc^2\dd\cc\dd\cc$ in $\Phi$, then $\gamma$ is obtained as follows:
We note that, for any squarefree $P$-module $M$ of dimension $d$,
the coefficient of $\cc^d$ in the $\bb$-expression is equal to $\dim_K M_\zero$.
Let $N'=\Omega(N^{(5)})/N^{(5)}$.
Then the coefficient of $\cc^2\dd\cc\dd\cc$ in $\Phi$ is the coefficient of of $\cc^2\dd\cc$
in the $\cc\dd$-index of $N'$,
and this is equal to the coefficient of $\cc^2$ in the $\cc\dd$-index of $N''=\Omega(N'^{(2)})/N'^{(2)}$,
which is equal to $\dim_K N_\zero''$.
This observation gives a ring-theoretic interpretation of \cite[Theorem 4.10]{Ka}.
\end{remark}

Here we give some examples of CW-posets whose $\aaa$-expression or $\bb$-expression is non-negative.

\begin{example}[Ehrenborg--Karu]
\label{ex1}
If $P$ is the face poset of a regular CW-complex $\Gamma$ which is homeomorphic to a ball then $\Omega(K[P])$ is isomorphic to an ideal of $K[P]$
generated by all $x_\sigma$ such that $\sigma$ is an interior face.
In this case, we have a natural injection $\Omega(K[P]) \to K[P]$.
Thus, the $\aaa$-expression of $\Psi_P(\aaa,\bb)$
is non-negative by Corollary \ref{4.5}(iii).
Moreover, if $\Phi+\Upsilon \aaa$ is the $\aaa$-expression,
then $\Upsilon$ is the $\cc\dd$-index of the face poset of the boundary of $\Gamma$.
See \cite[p.\ 231]{EK}.
\end{example}

\begin{example}
\label{ex2}
A Cohen--Macaulay quasi CW-poset $P$ is said to be {\em doubly Cohen--Macaulay} if, for any element $\sigma \in \widehat P$,
the poset $P \setminus \{\sigma\}$ is Cohen--Macaulay and has the same rank as $P$.
If $P$ is doubly Cohen--Macaulay
then there is an injection $K[P] \to \Omega(K[P])$ (see \cite[p.\ 91]{StG}).
Thus doubly Cohen--Macaulay CW-posets have non-negative $\bb$-expressions.
\end{example}

Now, we prove the main result of this section which implies Theorem \ref{main1} in the introduction.

\begin{theorem}
\label{4.9}
Let $M$ be a squarefree $P$-module of dimension $d$.
There are unique $\cc\dd$-polynomials $\Phi^\dd,\Phi^\aaa,\Phi^\bb \in \ZZ \langle \cc,\dd\rangle$ such that
\begin{align}
\label{4-4}
\Psi_M(\aaa,\bb)=\Phi^\dd \cdot \dd + \Phi^\aaa \cdot \aaa + \Phi^\bb \cdot\bb.
\end{align}
Moreover, if $M$ is Cohen--Macaulay then all the coefficients in $\Phi^\dd,\Phi^\aaa$ and $\Phi^\bb$ are non-negative.
\end{theorem}

\begin{proof}
We first prove the uniqueness.
If $\Psi_M(\aaa,\bb)$ can be written in the form $\Psi_M(\aaa,\bb)=\Phi^\dd \cdot \dd +\Phi^\aaa \cdot \aaa + \Phi^\bb \cdot \bb$ then
$(\Phi^\dd \cdot \dd +\Phi^\aaa \cdot \cc) + (\Phi^\bb -\Phi^\aaa) \cdot \bb$ is the $\bb$-expression
of $\Psi_M(\aaa,\bb)$.
Then the uniqueness of $\Phi^\dd,\Phi^\aaa$ and $\Phi^\bb$ follows from the uniqueness of the $\bb$-expression.

Next, we prove the existence and non-negativity.
Let $\Psi_M(\aaa,\bb)=\Phi + \Upsilon \bb$
be the $\bb$-expression of $\Phi_M(\aaa,\bb)$.
If we write $\Phi=\Phi' \cc + \Phi'' \dd$, then 
$$\Psi_M(\aaa,\bb) = \Phi'' \cdot \dd + \Phi' \cdot \aaa + (\Phi'+\Upsilon) \cdot \bb.$$
This proves the existence of \eqref{4-4}.
Also, Corollary \ref{4.5}(i) implies that $\Phi^\aaa=\Phi'$ and $\Phi^\dd=\Phi''$ are non-negative if $M$ is Cohen--Macaulay.
Finally, the non-negativity of $\Phi^\bb$ follows from the duality $\Psi_M(\aaa,\bb)=\Psi_{\Omega(M)}(\bb,\aaa)$.
\end{proof}

We call the right-hand side of \eqref{4-4} the {\em extended $\cc\dd$-index} of $M$ (or of $P$ if $M=K[P]$).
Note that, if $M$ has the symmetric flag $h$-vector, then $\Phi^\aaa=\Phi^\bb$ and $\Phi^\aaa \cc + \Phi^\dd \dd$
is the ordinal $\cc\dd$-index.

\begin{remark}
The existence of \eqref{4-4} holds for more general posets.
Indeed, to prove Lemma \ref{4.1}, it suffices to assume that $\partial \sigma$ has the $\cc\dd$-index for each $\sigma \in \widehat P$.
Thus, Lemma \ref{4.1} and the existence of \eqref{4-4} hold for posets $P$
such that $\langle \sigma \rangle$ is Eulerian \cite[Section 3.16]{St12} for all $\sigma \in P$.
\end{remark}

\begin{example}
Let $\Delta$ be the $2$-dimensional polyhedral complex obtained from the boundary of the square pyramid by gluing one triangle along an edge in the square
(see the following Figure).
\medskip

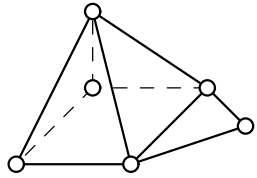
\begin{center}
\unitlength 0.1in
\begin{picture}( 12.8000,  8.8000)( 15.6000,-14.4000)
%
\special{pn 13}%
\special{pa 1600 1400}%
\special{pa 2200 1400}%
\special{fp}%
\special{pa 2200 1400}%
\special{pa 2600 1000}%
\special{fp}%
%
\special{pn 8}%
\special{pa 2600 1000}%
\special{pa 2000 1000}%
\special{da 0.070}%
\special{pa 2000 1000}%
\special{pa 1600 1400}%
\special{da 0.070}%
%
\special{pn 13}%
\special{pa 2000 600}%
\special{pa 1600 1400}%
\special{fp}%
\special{pa 2000 600}%
\special{pa 2200 1400}%
\special{fp}%
\special{pa 2000 600}%
\special{pa 2600 1000}%
\special{fp}%
%
\special{pn 8}%
\special{pa 2000 1000}%
\special{pa 2000 600}%
\special{da 0.070}%
%
\special{pn 13}%
\special{pa 2200 1400}%
\special{pa 2800 1200}%
\special{fp}%
\special{pa 2800 1200}%
\special{pa 2600 1000}%
\special{fp}%
%
\special{pn 13}%
\special{sh 0}%
\special{ar 2200 1400 40 40  0.0000000 6.2831853}%
%
\special{pn 13}%
\special{sh 0}%
\special{ar 2800 1200 40 40  0.0000000 6.2831853}%
%
\special{pn 13}%
\special{sh 0}%
\special{ar 2600 1000 40 40  0.0000000 6.2831853}%
%
\special{pn 13}%
\special{sh 0}%
\special{ar 2000 1000 40 40  0.0000000 6.2831853}%
%
\special{pn 13}%
\special{sh 0}%
\special{ar 2000 600 40 40  0.0000000 6.2831853}%
%
\special{pn 13}%
\special{sh 0}%
\special{ar 1600 1400 40 40  0.0000000 6.2831853}%
\end{picture}%
\end{center}
Let $P$ be the face poset of $\Delta$.
Then $P$ is Cohen--Macaulay, but is not Gorenstein*.
Its flag $f$-vector and the flag $h$-vector are given by
$$\sum_{S \subset [3]} f_S w_S= \aaa \aaa \aaa + 6 \bb \aaa \aaa + 10 \aaa \bb \aaa + 6 \aaa \aaa \bb + 20 \bb \bb \aaa + 19 \bb \aaa \bb +19 \aaa \bb \bb + 38 \bb \bb \bb$$
and
$$\Psi_P(\aaa,\bb)=\aaa \aaa \aaa + 5 \bb \aaa \aaa + 9 \aaa \bb \aaa + 5 \aaa \aaa \bb +5 \bb \bb \aaa + 8 \bb \aaa \bb +4 \aaa \bb \bb + \bb \bb \bb.$$
The extended $\cc\dd$-index of $P$ is
$$(4\cc) \dd + (\cc^2+4\dd)\aaa + (\cc^2+3\dd)\bb.$$
\end{example}

While the non-negativity of the extended $\cc\dd$-index of Cohen--Macaulay squarefree $P$-modules
easily follows from Karu's results,
it implies an interesting property on ordinal $h$-vectors.
For a squarefree $P$-module $M$ of dimension $d$,
the {\em $h$-vector} of $M$ is the vector $(h_0(M),h_1(M),\dots,h_d(M))$ defined by
$$h_i(M) = \sum_{S \subset [d], |S| =i} h_S(M).$$
By Lemma \ref{2.5}, this definition coincides with the usual definition of the $h$-vector.
The next corollary proves Corollary \ref{main2}.

\begin{corollary}
\label{4.9}
Let $M$ be a Cohen--Macaulay squarefree $P$-module of dimension $d$ and $h(M)=(h_0,h_1,\dots,h_d)$.
Then
\begin{itemize}
\item[(i)] $h_k \leq h_{d-1-k}$ and $h_{d-k} \leq h_{k+1}$ for all $0 \leq k < \frac d 2$.
\item[(ii)] $h_{k-1} \le h_k$ for $k \leq \frac d 2$ and $h_k \geq h_{k+1}$ for $k \geq \frac d 2$.
\end{itemize}
\end{corollary}

\begin{proof}
Let $\Psi_M(\aaa,\bb)=\Phi^\dd\cdot \dd + \Phi^\aaa \cdot \aaa + \Phi^\bb \cdot \bb$ be the extended $\cc\dd$-index of $M$.
Then, by substituting $\aaa=1$, one obtains
\begin{align}
\label{4-6}
\Psi_M(1,\bb)=\Phi^\dd(1+\bb,2\bb) \cdot 2\bb + \Phi^\aaa (1+\bb,2\bb) + \Phi^\bb (1+\bb,2\bb) \cdot \bb.
\end{align}
On the other hand, by the definition of the $h$-vector, one has
\begin{align}
\label{4-6-2}
\Psi_M(1,\bb) = h_0(M) + h_1 (M) \bb + \cdots +h_d(M) \bb^d.
\end{align}
For any homogeneous $\cc\dd$-polynomial $\Upsilon(\cc,\dd)\in \ZZ \langle \cc,\dd \rangle$ of degree $k$
whose coefficients are non-negative,
$\Upsilon(1+\bb,2\bb)$ can be written in the form
$$\Upsilon(1+\bb,2\bb)=\alpha_0 (1+\bb)^k + \alpha_1 \bb (1+\bb)^{k-2} + \alpha_2 \bb^2 (1+\bb)^{k-4} + \cdots$$
where $\alpha_0,\alpha_1,\alpha_2,\dots$ are non-negative integers.
Thus if we write 
$$\Upsilon(1+\bb,2\bb)=\gamma_0 + \gamma_1 \bb + \cdots + \gamma_k\bb^k,$$
then $(\gamma_0,\gamma_1,\dots,\gamma_k)$ is a symmetric vector satisfying $ \gamma_0 \leq \cdots \leq \gamma_{\frac k 2} \geq \cdots \geq \gamma_k$
when $k$ is even and $ \gamma_0 \leq \cdots \leq \gamma_{\frac {k-1} 2}=\gamma_{\frac {k+1} 2} \geq \cdots \geq \gamma_k$ when $k$ is odd.
By applying these facts to \eqref{4-6} and \eqref{4-6-2}
we obtain the desired property.
\end{proof}

\section{Lefschetz properties}

Recently, Kubitzke and Nevo \cite[Theorem 1.1]{KN} proved that if $\Delta$ is the barycentric subdivision of a shellable simplicial complex of dimension $d-1$,
then there is an l.s.o.p.\ $\Theta$ of $K[\Delta]$ and a linear form $w$ such that the multiplication map
$$\times w ^{d-1 -2k} : (K[\Delta]/(\Theta K[\Delta]))_k \to  (K[\Delta]/(\Theta K[\Delta]))_{d-1-k}$$
is injective for $k \leq \frac {d-1} 2$.
By using this result, it was proved in \cite[Corollary 1.3]{KN} that if $(h_0,h_1,\dots,h_d)$ is the $h$-vector of the barycentric subdivision of a Cohen--Macaulay simplicial complex of dimension $d-1$, then $(h_0,h_1-h_0,\dots,h_{\lfloor \frac d 2 \rfloor}-h_{\lfloor \frac d 2 \rfloor-1})$
is an $M$-vector, that is, the Hilbert function of a graded $K$-algebra.
The purpose of this section is to extend these results to barycentric subdivisions of Cohen--Macaulay polyhedral complexes.

We need the following fact which is a consequence of the Hard Lefschetz Theorem
\cite[III Theorems 1.3 and 1.4]{StG}.

\begin{lemma}
\label{5.1}
If $P$ is the face poset of a convex $d$-polytope,
then there is an l.s.o.p.\ $\Theta$ of $\RR[P]$ and a linear form $w \in \RR[P]$ such that
the multiplication map
$$\times w ^{d-2k} : (\RR[P]/(\Theta \RR[P]))_k \to (\RR[P]/(\Theta \RR[P]))_{d-k}$$
is bijective for $k \leq \frac d 2$.
\end{lemma}

\begin{proof}
Let $\rho$ be the maximal element of $P$ corresponding to the convex polytope itself,
and let $\partial P=P\setminus \{\rho\}$.
Then $\RR[P]=\RR[\partial P][x_\rho]$ and $\RR [\partial P]$ is the Stanley--Reisner ring of the barycentric subdivision of the boundary of a convex polytope.
Since the barycentric subdivision of a convex polytope can be regarded as a convex polytope (see \cite{ES}),
it follows from \cite[III Theorem 1.3]{StG} that, by an appropriate choice of an l.s.o.p.\ $\Theta$ of $\RR[\partial P]$,
$\RR[\partial P]/(\Theta \RR[\partial P])$ is isomorphic to the cohomology ring of a toric variety arising from an integral simplicial $d$-polytope.
Since $x_\rho,\Theta$ is an l.s.o.p.\ of $\RR[P]$ and
$\RR[P]/((x_\rho,\Theta)\RR[P])\cong \RR[\partial P]/(\Theta \RR[\partial P])$,
the desired statement follows from \cite[III Theorem 1.4]{StG}.
\end{proof}

We call a linear form $w$ in Lemma \ref{5.1} a {\em Lefschetz element} of $\RR[P]/(\Theta \RR[P])$.

Recall that a CW-poset $P$ is said to be of {\em polyhedral type} if,
for any $\sigma \in P$, $\langle \sigma \rangle$ is the face poset of a convex polytope.
Obviously, the face poset of a polyhedral complex is a CW-poset of polyhedral type.

\begin{theorem}
\label{5.2}
Let $P$ be a CW-poset of polyhedral type,
and let $M$ be a Cohen--Macaulay squarefree $P$-module over $\RR$ of dimension $d$.
There is an l.s.o.p.\ $\Theta$ of $M$ and a linear form $w$ such that
\begin{itemize}
\item[(i)] the multiplication
$$\times w^{d-1-2k}: (M/\Theta M)_k \to (M/\Theta M)_{d-1-k}$$
is injective for $k \leq \frac {d-1} 2$.
\item[(ii)] the multiplication
$$\times w^{d-1-2k}: (M/\Theta M)_{k+1} \to (M/\Theta M)_{d-k}$$
is surjective for $k \leq \frac {d-1} 2$.
\end{itemize}
\end{theorem}

\begin{proof}
Let $R=\RR[x_\sigma:\sigma \in \widehat P]$.
We prove that, for a general choice of $\Theta$ and $w$,
conditions (i) and (ii) hold.
Since, for a graded $R$-module $N$ and a linear form $w \in R$,
the multiplication $\times w : N_k \to N_{k+1}$ is surjective if and only if
$\times w: (N^T)_{-k-1} \to (N^T)_{-k}$ is injective,
by Lemma \ref{3.3} it suffices to prove (ii).

For a linear form $\theta=\sum_{\sigma \in \widehat P} \alpha_\sigma x_\sigma \in R$,
let $\theta^{\leq k} = \sum_{\rank(\sigma) \leq k} \alpha_\sigma x_\sigma$.
It follows from \cite[Proposition 3.6]{Sw} that, if we take sufficiently general linear forms $\theta_1,\dots,\theta_{d+1}$,
then they satisfy that, for each $\sigma \in \widehat P$ with $\rank \sigma =r$, 
$\theta_1^{\leq r},\dots,\theta_{r}^{\leq r}$ is an l.s.o.p.\ of $\RR[\langle \sigma \rangle]$
and $\theta_{r+1}^{\leq r}$ is a Lefschetz element of $\RR[\langle \sigma \rangle]/((\theta_1^{\leq r},\dots,\theta_{r}^{\leq r}) \RR[\langle \sigma \rangle])$.

Let $\Theta=\theta_1,\dots,\theta_d$.
Consider the submodule $N= \bigoplus_{\sigma \in P_d} M_{\ee_\sigma} R \subset M$.
Since $M_{\ee_\sigma}R \cong M_{\ee_\sigma} \otimes_\RR \RR[\langle \sigma \rangle]$ for $\sigma \in P_d$,
we have
$$N/(\Theta N) \cong \bigoplus_{\sigma \in P_d}\big(  M_{\ee_\sigma} \otimes_\RR \big(\RR[\langle \sigma \rangle]/(\Theta \RR[\langle \sigma \rangle]) \big) \big).$$
(Here elements of $M_{\ee_\sigma}$ have degree $\ee_\sigma$.)
Thus the multiplication
$$\times \theta_{d+1}^{d-1-2k}: (N/\Theta N)_{k+1} \to (N/\Theta N)_{d-k}$$
is bijective for $k \leq \frac {d-1} 2$.
Consider the following commutative diagram
\begin{eqnarray*}
\begin{array}{rclcc}
(N/(\Theta N))_{k+1} &\longrightarrow & (M/(\Theta M))_{k+1}\medskip\\
\times \theta_{d+1}^{d-1-2k} \downarrow\ \ \ \ \ \ \ \ \  & & \ \ \ \ \ \ \ \downarrow \times \theta_{d+1}^{d-1-2k}\medskip\\
(N/(\Theta N))_{d-k} &\longrightarrow & (M/(\Theta M))_{d-k}.
\end{array}
\end{eqnarray*}
To prove the surjectivity of the right vertical map,
it is enough to prove that the lower horizontal map is surjective.
Thus, we prove that the elements in $N_k$ generates $(M/(\Theta M))_k$
for $k \geq \frac {d+1}  2$.
For $\uu =\sum_{\sigma \in \widehat P} u_\sigma \ee_\sigma \in \NN^{|\widehat P|}$,
we write $\rank (\uu)=\rank(\max(\supp(\uu)))$ and $|\uu|=\sum_{\sigma \in \widehat P} u_\sigma$.
Let $\mu \in M_{\uu}$ with $\rank(\uu)=r<d$ and with $|\uu| \geq \frac{d+1} 2$.
We claim that
$$\mu \in \Theta M + \bigoplus_{\rank(\vv)>r} M_\vv.$$
Note that this claim implies that $N_k$ generates $(M/(\Theta M))_k$ for $k \geq \frac {d+1} 2$.

Let $\sigma = \max(\supp(\uu))$.
By condition (b') of squarefree $P$-modules,
$\mu=\tilde \mu x^{\uu -\ee_\sigma}$ for some $\tilde \mu \in M_{\ee_\sigma}$
and
$$R/ (\mathrm{ann}\ \! \tilde \mu +(x_\tau : \rank \tau >r)) \cong \RR[\langle \sigma \rangle],$$
where $\mathrm{ann}\ \! \tilde \mu=\{f \in R: f\tilde \mu=0\}$.
Since $\theta_{r+1}^{\leq r}$ is a Lefschetz element of 
$$\RR[\langle \sigma \rangle]/\big((\theta_1^{\leq r},\dots,\theta_r^{\leq r})\RR[\langle \sigma \rangle]\big)
\cong R/(\mathrm{ann}\ \! \tilde \mu +(\theta_1^{\leq r},\dots,\theta_r^{\leq r})+(x_\tau : \rank \tau >r)),$$
we have 
$$(\mathrm{ann}\ \! \tilde \mu +(\theta_1^{\leq r},\dots,\theta_{r+1}^{\leq r})+(x_\tau : \rank \tau >r))_k=R_k$$ for $k \geq \frac {r} 2$.
Since $r <d$, we have $\deg x^{\uu-\ee_\sigma} \geq \frac {d-1} 2 \geq \frac  r 2$.
Thus we have
$$x^{\uu-\ee_\sigma} \in \mathrm{ann}\ \! \tilde \mu +(\theta_1^{\leq r},\dots,\theta_{r+1}^{\leq r})+(x_\tau : \rank \tau >r).$$
Hence
\begin{align*}
\mu = \tilde \mu x^{\uu-\ee_\sigma} \in \left((\theta_1^{\leq r},\dots, \theta_{r+1}^{\leq r})+(x_\tau : \rank \tau >r)\right)\tilde \mu 
 \subset \Theta M + \bigoplus_{\rank(\vv) >r} M_\vv,
\end{align*}
as desired.
\end{proof}

A Cohen--Macaulay graded $R$-module $M$ of dimension $d$ is said to have the {\em weak Lefschetz property} (WLP for short)
if there is an l.s.o.p.\ $\Theta$ of $M$ and a linear form $w \in R$ such that the multiplication $\times w : (M/(\Theta M))_{k-1} \to (M/\Theta M))_{k}$
is either injective or surjective for all $k$.
Theorem \ref{5.2} implies the following corollary.

\begin{corollary}
\label{5.3}
Let $P$ be a CW-poset of polyhedral type and $M$ a Cohen--Macaulay squarefree $P$-module of dimension $d$ over $\RR$.
There is an l.s.o.p.\ $\Theta$ of $M$ and a linear form $w$ such that the multiplication
$\times w : (M/(\Theta M))_{k-1} \to (M/\Theta M)_{k}$ is injective for $k \leq \frac d 2$
and surjective for $k \geq \frac d 2+1$.
In particular, $M$ has the WLP if $d$ is even.
\end{corollary}

\begin{remark}
Theorem \ref{5.2} and Corollary \ref{5.3} hold over any infinite field if $\langle \sigma \rangle$ is the face poset of a simplex for any $\sigma \in \widehat P$
since Lemma \ref{5.1} holds for simplices over any infinite field \cite[Proposition 2.3]{KN}.
Thus, for barycentric subdivisions of simplicial complexes, one can work over positive characteristic.
In particular, Corollary \ref{5.3} solves the conjecture of Kubitzke and Nevo \cite[Conjecture 4.12]{KN}
in odd dimensions.
\end{remark}

Note that Corollary \ref{5.3} cannot prove the WLP when $d$ is odd since it says nothing about the 
multiplication map $\times w : (M/(\Theta M))_{\frac {d-1} 2} \to (M/\Theta M))_{\frac {d+1} 2}$.

Finally, we prove Theorem \ref{main3}.

\begin{proof}[Proof of Theorem \ref{main3}]
Let $P$ be a Cohen--Macaulay CW-poset of polyhedral type having rank $n$,
and let $(h_0,h_1,\dots,h_n)$ be the $h$-vector of $\Delta_P$.
By Corollary \ref{5.3}, there is an l.s.o.p.\ $\Theta=\theta_1,\dots,\theta_n$ of $\RR[P]$ and a linear form $w$ such that
\begin{align*}
&\dim_\RR \big(\RR[P]/((w,\Theta)\RR[P])\big)_k \\
&=\dim_\RR\big(\RR[P]/(\Theta\RR[P])\big)_k -
\dim_\RR\big(\RR[P]/(\Theta\RR[P])\big)_{k-1}\\
&=h_k -h_{k-1}
\end{align*}
for $k \leq \frac d 2$,
where the second equality follows since $h_i = \dim_\RR (\RR[P]/(\Theta\RR[P]))_i$ for all $i$ (see e.g., \cite[II Corollary 2.5]{StG}).
We prove that the Hilbert function of $\RR[P]/((w,\Theta)\RR[P])$
is the $f$-vector of a simplicial complex.

Let $\RR[P]=\RR[x_\sigma:\sigma \in \widehat P]/I$
and $c=|\widehat P|-n$.
Since $I$ is generated by monomials of degree $2$,
by \cite[Theorem 2.1]{CCV} the ideal $I$ contains a regular sequence of the form
$\ell_1\ell_2,\dots,\ell_{2c-1} \ell_{2c}$, where $\ell_1,\dots,\ell_{2c}$ are linear forms.
Then, if we choose $\Theta$ sufficiently general,
$\Lambda =\ell_1\ell_2,\dots,\ell_{2c-1} \ell_{2c},\theta_1,\dots,\theta_n$
is also a regular sequence.
Since $\RR[P]/((w,\Theta)\RR[P])=\RR[x_\sigma:\sigma\in \widehat P]/(I+(w,\Theta))$
and $I+(w,\Theta)$ contains $\Lambda$,
the desired statement follows from Abedelfatah's result on the Eisenbud-Green-Harris conjecture \cite[Corollary 4.3]{Ab}.
\end{proof}

\section{Upper bounds for the $\cc\dd$-indices}

In this section, we study upper bounds of the $\cc\dd$-indices of Gorenstein* posets.
Billera and Ehrenborg \cite{BE} proved that the $\cc\dd$-index of a $d$-polytope with $v$ vertices are bounded above by the $\cc\dd$-index
of the cyclic $d$-polytope with $v$ vertices.
Reading \cite[Section 7]{Rea} study upper bounds of the $\cc\dd$-indices of Bruhat intervals in terms of the length of an interval.
It is not possible to obtain upper bounds of the $\cc\dd$-indices of Gorenstein* posets
for a fixed number of rank $1$ elements or for a fixed rank
since most coefficients of the $\cc\dd$-indices of Gorenstein* posets can be arbitrary large
even if we fix their rank and the number of rank $1$ elements.
However, if we fix the number of rank $i$ elements for all $i$,
the $\cc\dd$-index is clearly bounded since its flag $f$-vector is bounded.
The purpose of this section is to find sharp upper bounds of the $\cc\dd$-indices of Gorenstein* posets
when we fix the number of rank $i$ elements for all $i$.

We first study some algebraic properties of squarefree $P$-modules.
We say that a squarefree $P$-module is {\em standard} if it is generated by elements of degree $0$.

\begin{lemma}
\label{6.1}
Let $P$ be a quasi CW-poset and $M$ a Cohen--Macaulay squarefree $P$-module of dimension $d$.
Then, for any $k<d-1$, the module $\Omega(M^{(k)})$ is standard.
\end{lemma}

\begin{proof}
Let
$$\mathcal L_\bullet^M :\  0 \longrightarrow \mathcal L_d^M \stackrel {\partial_d} \longrightarrow \mathcal L_{d-1} ^M  \stackrel {\partial_{d-1}} \longrightarrow \mathcal L_{d-1} ^M \stackrel {\partial_{d-2}} \longrightarrow \cdots$$
be the Karu complex of $M$.
Observe that each $\Image \partial_k$ is standard since $\mathcal L_k^M$ is the direct sum of Stanley--Reisner rings.
Then the desired statement follows since 
$\Omega(M^{(k)}) = \ker \partial_{k+1} \cong \Image \partial_{k+2}$
by Theorem \ref{3.4}.
\end{proof}

\begin{remark}
\label{2CM}
For a Cohen--Macaulay simplicial complex $\Delta$,
the module $\Omega(K[\Delta])$ is generated in degree $0$ if and only if $\Delta$ is doubly Cohen--Macaulay \cite[III Section 3]{StG} over $K$.
It is known that a proper skeleton of a Cohen--Macaulay CW-poset satisfying the intersection property is always doubly Cohen--Macaulay \cite[Corollary 2.5]{Fl}.
(It is noteworthy that his definition of the doubly Cohen--Macaulay property is slightly different from ours, and ``the intersection property" is really necessary in his context.)
Lemma \ref{6.1} is an analogue of this fact for squarefree $P$-modules.
\end{remark}

For a squarefree $P$-module $M$,
we write $\Phi_M^\dd\cdot \dd+\Phi_M^\aaa\cdot \aaa+\Phi_M^\bb\cdot \bb$ for its extended $\cc\dd$-index.
Also, we write $\Phi_P^\bullet=\Phi_{K[P]}^\bullet$, where $\bullet$ is $\aaa,\bb$ or $\dd$.

\begin{lemma}
\label{a7.3}
Let $P$ be a Cohen--Macaulay quasi CW-poset and $M$ a Cohen--Macaulay standard squarefree $P$-module with $\dim M=\rank P$.
Then we have the coefficients inequality 
$$\Phi_M^\dd\cdot \dd+\Phi_M^\aaa\cdot \aaa+\Phi_M^\bb\cdot \bb \leq
(\dim_K M_\zero )(\Phi_P^\dd\cdot \dd+\Phi_P^\aaa\cdot \aaa+\Phi_P^\bb\cdot \bb).$$
\end{lemma}

\begin{proof}
Let $c=\dim_K M_\zero$ and let
$N=\bigoplus_{i=1}^c K[P]$ be the direct sum of $c$ copies of $K[P]$.
Since $M$ is standard, there is a surjection $\pi:N \twoheadrightarrow M$.
Then we have the following short exact sequence
$$0 \longrightarrow \ker \pi \longrightarrow N \stackrel \pi \longrightarrow M \longrightarrow 0.$$
Observe that $N$ and $M$ are Cohen--Macaulay squarefree $P$-modules having the same dimension.
It is clear that $\dim\ \!\! (\ker \pi) \leq \dim M$.
Also, since the depth of $\ker \pi$ is larger than or equal to the depth of $M$, which is equal to $\dim M$ since $M$ is Cohen--Macaulay, by the depth lemma \cite[Proposition 1.2.9]{BH}
and since the depth is smaller than or equal to the dimension \cite[Proposition 1.2.12]{BH},
it follows that
$\ker \pi$ is a Cohen--Macaulay squarefree $P$-module with $\dim \ \!\!(\ker \pi)= \dim M$.
Since the short exact sequence says that the extended $\cc\dd$-index of $N$ is the sum of
those of $\ker \pi$ and $M$,
the desired statement follows from the non-negativity of the extended $\cc\dd$-index.
\end{proof}

Lemma \ref{a7.3} has the following combinatorial meanings,
which says that the extended $\cc\dd$-index of a Cohen--Macaulay regular CW-complex
is larger than or equal to that of its Cohen--Macaulay subcomplex having the same dimension.

\begin{corollary}
\label{a7.4}
Let $P$ be a Cohen--Macaulay quasi CW-poset and $Q$ an order ideal of $P$ having the same rank as $P$.
If $Q$ is Cohen--Macaulay,
then we have the coefficients inequality 
$$\Phi_Q^\dd\cdot \dd+\Phi_Q^\aaa\cdot \aaa+\Phi_Q^\bb\cdot \bb \leq
\Phi_P^\dd\cdot \dd+\Phi_P^\aaa\cdot \aaa+\Phi_P^\bb\cdot \bb.$$
\end{corollary}

\begin{proof}
Since $K[Q]$ is a standard squarefree $P$-module,
the statement is the special case of Lemma \ref{a7.3} when $M=K[Q]$.
\end{proof}

For a homogeneous $\cc\dd$-polynomial $\Phi=\sum_{v \in \mathcal B_d} \alpha_v v \in \ZZ \langle \cc,\dd\rangle$ of degree $d$
and for a $\cc\dd$-monomial $u$ of degree $m<d$,
let
$$\Phi_u=\sum_{v \in \mathcal B_{d-m}} \alpha_{v  u} v.$$
Note that $\Phi_{\dd\cc^k}$ is equal to $\Phi_{d-k-2}$ of \eqref{4-2} in Section 5.

\begin{lemma}
\label{6.3}
Let $P$ be a quasi CW-poset, $M$ a Cohen--Macaulay squarefree $P$-module of dimension $d$ over $\RR$
and let $\Psi_M(\aaa,\bb)=\Phi + \Upsilon \bb$ be the $\bb$-expression of $\Psi_M(\aaa,\bb)$.
For any $\cc\dd$-monomial of degree $\leq d$ of the form $u=\dd u'$, there is a Cohen--Macaulay standard squarefree $P$-module $N$ such that
$\Psi_N(\aaa,\bb)=\Phi_u$.
\end{lemma}

\begin{proof}
We may assume that $u$ is of the form $u=\dd\cc^k$.
Then $\Phi_u$ is the $\cc\dd$-index of $N=\Omega(M^{(d-k-2)})/M^{(d-k-2)}$ by Corollary \ref{4.4} and $N$ is standard by Lemma \ref{6.1}.
\end{proof}

For a squarefree $P$-module $M$ of dimension $d$ with the $\bb$-expression $\Psi_M(\aaa,\bb)=\Phi+ \Upsilon \bb$,
we write
$$\alpha_S(M)=\alpha_S(\Phi)$$
for $S \in \mathcal A_d$, where $\alpha_S(\Phi)$ is as in the introduction.
Also, we write $\alpha_S(P)=\alpha_S(K[P])$.
Note that, for Gorenstein* posets,
this definition coincides with that given in the introduction since the $\bb$-expression of a Gorenstein* poset
is equal to its $\cc\dd$-index.

\begin{lemma}
\label{6.4}
Let $P$ be a Cohen--Macaulay quasi CW-poset
and $M$ a Cohen--Macaulay standard squarefree $P$-module of dimension $d$.
Then, for any $S \in \mathcal A_d$,
one has
$$\alpha_{S}(M) \leq \alpha_\emptyset (M) \alpha_{S}(P).$$
\end{lemma}

\begin{proof}
Since $M$ is a squarefree $P^{(d-1)}$-module and since $\alpha_S(P^{(d-1)})=\alpha_S(P)$ for all $S \in \mathcal A_d$
by the computation given before Corollary \ref{4.4},
we may assume that $\dim M=\rank P$.
Since $(\Phi^\aaa_M \cdot \cc + \Phi_M^\dd \cdot \dd) + (\Phi_M^\bb - \Phi_M^\aaa)\cdot \bb$ is the $\bb$-expression of $\Psi_M(\aaa,\bb)$,
the desired statement follows from Lemma \ref{a7.3}.
\end{proof}

Now we prove the main result of this section.

\begin{theorem}
\label{6.5}
Let $P$ be a Cohen--Macaulay quasi CW-poset.
If $M$ is a Cohen--Macaulay standard squarefree $P$-module over $\RR$ of dimension $d$,
then $\alpha_S(M) \leq \alpha_\emptyset(M) \prod_{i \in S} \alpha_{\{i\}}(P)$
for all $S \in \mathcal A_d$.
\end{theorem}

\begin{proof}
We prove the statement by induction on $|S|$.
If $|S|=0$, then the statement is obvious.
Let $S \in \mathcal A_n$ with $|S| \geq 1$
and let $a = \min S$.
By the induction hypothesis, $\alpha_{S \setminus \{a\}} (M) \leq \alpha_\emptyset(M) \prod _{i \in S \setminus \{a\}} \alpha_{\{i\}}(P)$.
By Lemma \ref{6.3}, there is a Cohen--Macaulay standard squarefree $P$-module $N$
such that $\alpha_\emptyset(N)=\alpha_{S \setminus \{a\}}(M)$
and $\alpha_{\{a\}}(N)=\alpha_S(M)$.
Then Lemma \ref{6.4} says 
$$\alpha_S(M) = \alpha_{\{a\}}(N) \leq \alpha_\emptyset(N) \alpha_{\{a\}}(P)
= \alpha_{S \setminus \{a\}}(M) \alpha_{\{a\}}(P) \leq \alpha_\emptyset(M) \prod_{i \in S} \alpha_{\{i\}}(P),$$
as desired.
\end{proof}

By considering the special case of Theorem \ref{6.5} when $M=K[P]$,
we obtain the following corollary which proves Theorem \ref{main4}.

\begin{corollary}
\label{6.6}
If $P$ is a Cohen--Macaulay quasi CW-poset rank $n$,
then we have $\alpha_S(P) \leq \prod_{i \in S} \alpha_{\{i\}}(P)$ for all $S \in \mathcal A_n$.
\end{corollary}

Recall that the {\em rank generating function} of a graded poset $P=\bigcup_{i=0}^n P_i$
is the polynomial $\sum_{k=0}^n f_{\{k\}}(P) t^k$, where $f_{\{0\}}(P)=1$.
For a Gorenstein* poset $P$ of rank $n$,
one has $f_{\{1\}}(P)-1=h_{\{1\}}(P)=1+\alpha_{\{1\}}(P)$
and 
$$f_{\{i\}}(P)-1=h_{\{i\}}(P)=1+\alpha_{\{i-1\}}(P) +\alpha_{\{i\}}(P)$$
for $i=2,3,\dots,n-1$.
This implies
$$\alpha_{\{k\}}(P)=-1 + \sum_{i=0}^k (-1)^{k-i} f_{\{i\}} (P)$$
for $k=1,2,\dots,n-1$.
These equations and the Euler relation $\sum_{i=0}^n (-1)^{n-i} f_{\{i\}}(P)=1$ say that knowing the integers $\alpha_{\{1\}}(P),\dots,\alpha_{\{n-1\}}(P)$
is equivalent to knowing $f_{\{1\}}(P),\dots,f_{\{n\}}(P)$.
Thus Corollary \ref{6.6} gives upper bounds of the $\cc\dd$-indices of Gorenstein* posets for a fixed rank generating function.
In the rest of this section, we prove that the bounds are sharp.
To prove this we use the technique, which was called {\em unzipping} in \cite{MN}.

Let $P$ be a Gorenstein* poset and $\sigma> \tau$ a cover relation with $\tau \ne \hat 0$.
We define the poset $\mathcal U (P;\sigma,\tau)$ as follows:
delete the cover relation $\sigma>\tau$, add elements $\sigma',\tau'$ with $\rank(\sigma')=\rank(\sigma), \rank(\tau')=\rank(\tau)$
and add cover relations (i) $\sigma'<\rho$ for all cover relations $\sigma<\rho$, (ii) $\rho<\tau'$ for all cover relations $\rho<\tau$ and (iii) $\tau'<\sigma', \tau<\sigma'$ and $ \tau'<\sigma$.
The following result was shown in \cite[Theorem 4.6]{Rea} and in \cite[Corollary 2.6]{MN}.

\begin{lemma}
\label{6.7}
Let $P$ be a Gorenstein* poset of rank $n$ and $\sigma> \tau$ a cover relation.
Then $\mathcal U (P;\sigma,\tau)$ is Gorenstein* and
$$\Phi_{\mathcal U (P;\sigma,\tau)}(\cc,\dd)=\Phi_P(\cc,\dd) + \Phi_{ \partial \tau} (\cc,\dd) \cdot \dd \cdot \Phi_{\lk_P(\sigma)}(\cc,\dd).$$
\end{lemma}

The next proposition guarantees that the bounds in Theorem \ref{main4} are sharp.

\begin{proposition}
\label{6.8}
For any sequence $\alpha_1,\dots,\alpha_{n-1}$ of nonnegative integers,
there is a Gorenstein* poset $P$ of rank $n$ such that 
\begin{itemize}
\item[(i)] $\alpha_S(P)= \prod_{i \in S} \alpha_i$ for all $S \in \mathcal A_n$.
\item[(ii)] there is $\sigma \in P_{n}$ such that $\alpha_S(\partial \sigma)=\prod_{i \in S} \alpha_i$ for all $S \in \mathcal A_{n-1}$.
\end{itemize}
\end{proposition}

\begin{proof}
We use induction on $n$.
The statement is obvious when $n=1$.
Suppose $n>1$.
By the induction hypothesis,
there is a Gorenstein* poset $Q$ of rank $n-1$ and $\tau \in Q_{n-1}$
such that $\alpha_S(Q)=\prod_{i \in S} \alpha_i$ for all $S \in \mathcal A_{n-1}$ and $\alpha_S(\partial \tau)= \prod_{i \in S} \alpha_i$
for all $S\in \mathcal A_{n-2}$.
Let $\Sigma Q= P \cup \{\eta,\eta'\}$ be the suspension of $P$.
Thus $\Sigma Q$ is the poset whose order is obtained from that of $Q$ by adding the relations $\eta> \rho$ and $\eta' >\rho$ for all $\rho \in Q$.
By \cite[Lemma 1.1]{Stcd},
$\Sigma Q$ is a Gorenstein* poset with $\Phi_{\Sigma Q}(\cc,\dd)=\Phi_{Q}(\cc,\dd)\cdot \cc$.

If $\alpha_{n-1}=0$
then the poset $\Sigma Q$ satisfies the desired conditions (i) and (ii) since $\partial \eta'=Q$.
Suppose $\alpha_{n-1}>0$.
Let $P(1)=\mathcal U(\Sigma Q;\eta,\tau)$ and let $\sigma(1) \in P(1)_n$ and $\tau{(1)}\in P(1)_{n-1}$
be the elements which are not in $\Sigma Q$.
For $k=2,3,\dots,\alpha_{n-1}$, we recursively define the poset 
$P(k)= \mathcal U(P(k-1); \sigma(k-1),\tau(k-1))$ and elements $\sigma(k) \in P(k)_n$ and $\tau(k) \in P(k)_{n-1}$ so that $\sigma(k)$ and $\tau(k)$ are the elements  which are not in $P(k-1)$.
We claim that $P=P(\alpha_{n-1})$ satisfies the desired conditions.
By the construction of $P(k)$,
$\partial \tau(k) =\partial \tau=\{ \rho \in P: \rho < \tau\}$ in $P(k)$.
Thus, by Lemma \ref{6.7}, we have
$$
\Phi_{P}=\Phi_{\Sigma Q} + \alpha_{n-1} \Phi_{\partial \tau}\cdot \dd
=\Phi_Q \cdot \cc + \alpha_{n-1} \Phi_{\partial \tau}\cdot \dd.
$$
Hence, for $S \in \mathcal A_n$,
\begin{eqnarray*}
\alpha_S(P)= \left\{
\begin{array}{llll}
\alpha_S(Q), & \mbox{ if } n-1 \not \in S,\\
 \alpha_{n-1}\cdot \alpha_{S\setminus \{n-1\}}(\partial \tau), & \mbox{ if }  n-1 \in S.
\end{array}
\right.
\end{eqnarray*}
By the assumption on $Q$ and $\tau$,
it follows that $P$ satisfies condition (i).
Also, since $\partial \eta'=Q$ in $P$,
$P$ also satisfies condition (ii).
\end{proof}

\begin{example}
The Gorenstein* poset given in the proof of Proposition \ref{6.8} is obtained from the poset of the zero dimensional sphere (that is, the CW-complex consisting of two vertices)
by taking suspensions and unzipping repeatedly.
For example, if $\alpha_1=\alpha_3=1$ and $\alpha_2=2$,
then we obtain the following poset.
\bigskip

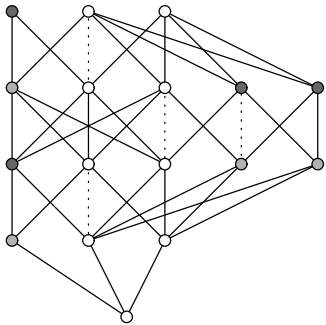
\begin{center}
\unitlength 0.1in
\begin{picture}( 16.6000, 16.6000)( 13.7000,-28.3000)
%
\special{pn 8}%
\special{pa 2000 2800}%
\special{pa 2200 2400}%
\special{fp}%
\special{pa 1800 2400}%
\special{pa 2000 2800}%
\special{fp}%
\special{pa 1400 2400}%
\special{pa 2000 2800}%
\special{fp}%
%
\special{pn 8}%
\special{pa 2200 2400}%
\special{pa 2200 2000}%
\special{fp}%
\special{pa 2200 2400}%
\special{pa 2600 2000}%
\special{fp}%
\special{pa 2200 2400}%
\special{pa 3000 2000}%
\special{fp}%
\special{pa 2200 2400}%
\special{pa 1800 2000}%
\special{fp}%
\special{pa 1800 2400}%
\special{pa 3000 2000}%
\special{fp}%
\special{pa 2600 2000}%
\special{pa 1800 2400}%
\special{fp}%
\special{pa 2200 2000}%
\special{pa 1800 2400}%
\special{fp}%
\special{pa 1800 2400}%
\special{pa 1400 2000}%
\special{fp}%
\special{pa 1400 2000}%
\special{pa 1400 2400}%
\special{fp}%
\special{pa 1400 2400}%
\special{pa 1800 2000}%
\special{fp}%
%
\special{pn 8}%
\special{pa 1800 2000}%
\special{pa 1800 2400}%
\special{dt 0.045}%
%
\special{pn 8}%
\special{pa 2600 2000}%
\special{pa 2600 1600}%
\special{dt 0.045}%
\special{pa 2200 1600}%
\special{pa 2200 2000}%
\special{dt 0.045}%
\special{pa 1800 1600}%
\special{pa 1800 1200}%
\special{dt 0.045}%
%
\special{pn 8}%
\special{pa 2200 1600}%
\special{pa 1800 2000}%
\special{fp}%
\special{pa 2200 1600}%
\special{pa 1400 2000}%
\special{fp}%
\special{pa 1800 1600}%
\special{pa 2200 2000}%
\special{fp}%
\special{pa 1800 2000}%
\special{pa 1800 1600}%
\special{fp}%
\special{pa 1800 1600}%
\special{pa 1400 2000}%
\special{fp}%
\special{pa 1400 1600}%
\special{pa 1400 2000}%
\special{fp}%
\special{pa 1400 1600}%
\special{pa 1800 2000}%
\special{fp}%
\special{pa 1400 1600}%
\special{pa 2200 2000}%
\special{fp}%
\special{pa 2200 1600}%
\special{pa 2600 2000}%
\special{fp}%
\special{pa 2200 2000}%
\special{pa 2600 1600}%
\special{fp}%
\special{pa 2600 1600}%
\special{pa 3000 2000}%
\special{fp}%
\special{pa 2600 2000}%
\special{pa 3000 1600}%
\special{fp}%
\special{pa 3000 1600}%
\special{pa 3000 2000}%
\special{fp}%
%
\special{pn 8}%
\special{pa 2200 1200}%
\special{pa 3000 1600}%
\special{fp}%
\special{pa 2200 1200}%
\special{pa 2600 1600}%
\special{fp}%
\special{pa 2200 1600}%
\special{pa 2200 1200}%
\special{fp}%
\special{pa 2200 1200}%
\special{pa 1800 1600}%
\special{fp}%
\special{pa 1800 1200}%
\special{pa 3000 1600}%
\special{fp}%
\special{pa 2600 1600}%
\special{pa 1800 1200}%
\special{fp}%
\special{pa 2200 1600}%
\special{pa 1800 1200}%
\special{fp}%
\special{pa 1800 1200}%
\special{pa 1400 1600}%
\special{fp}%
\special{pa 1400 1600}%
\special{pa 1400 1200}%
\special{fp}%
\special{pa 1400 1200}%
\special{pa 1800 1600}%
\special{fp}%
%
\special{pn 8}%
\special{sh 0.600}%
\special{ar 1400 1200 30 30  0.0000000 6.2831853}%
%
\special{pn 8}%
\special{sh 0}%
\special{ar 1800 1200 30 30  0.0000000 6.2831853}%
%
\special{pn 8}%
\special{sh 0}%
\special{ar 2200 1200 30 30  0.0000000 6.2831853}%
%
\special{pn 8}%
\special{sh 0.600}%
\special{ar 3000 1600 30 30  0.0000000 6.2831853}%
%
\special{pn 8}%
\special{sh 0.600}%
\special{ar 2600 1600 30 30  0.0000000 6.2831853}%
%
\special{pn 8}%
\special{sh 0}%
\special{ar 2200 1600 30 30  0.0000000 6.2831853}%
%
\special{pn 8}%
\special{sh 0}%
\special{ar 1800 1600 30 30  0.0000000 6.2831853}%
%
\special{pn 8}%
\special{sh 0.300}%
\special{ar 1400 1600 30 30  0.0000000 6.2831853}%
%
\special{pn 8}%
\special{sh 0.600}%
\special{ar 1400 2000 30 30  0.0000000 6.2831853}%
%
\special{pn 8}%
\special{sh 0}%
\special{ar 1800 2000 30 30  0.0000000 6.2831853}%
%
\special{pn 8}%
\special{sh 0}%
\special{ar 2200 2000 30 30  0.0000000 6.2831853}%
%
\special{pn 8}%
\special{sh 0.300}%
\special{ar 2600 2000 30 30  0.0000000 6.2831853}%
%
\special{pn 8}%
\special{sh 0.300}%
\special{ar 3000 2000 30 30  0.0000000 6.2831853}%
%
\special{pn 8}%
\special{sh 0}%
\special{ar 2200 2400 30 30  0.0000000 6.2831853}%
%
\special{pn 8}%
\special{sh 0}%
\special{ar 1800 2400 30 30  0.0000000 6.2831853}%
%
\special{pn 8}%
\special{sh 0.300}%
\special{ar 1400 2400 30 30  0.0000000 6.2831853}%
%
\special{pn 8}%
\special{sh 0}%
\special{ar 2000 2800 30 30  0.0000000 6.2831853}%
\end{picture}%
\end{center}
Dotted lines are the relations which are removed by unzipping and colored elements are those which are added by unzipping.
Black elements corresponds to $\sigma(-)$ and brown elements corresponds to $\tau(-)$ in the proof.
\end{example}

\section{Concluding remarks}

\subsection*{On $f$-vectors}
Recall that the $f$-vector $f(\Delta_P)=(f_{-1},f_0,\dots,f_{n-1})$ of the order complex $\Delta_P$ of a poset $P$ of rank $n$
is given by $f_{i-1}=\sum_{S \subset [n], |S|=i} f_S(P)$.
Considering Corollary \ref{main2}, we ask the following question.

\begin{problem}
Is the $f$-vector of the order complex of a Cohen--Macaulay quasi CW-poset unimodal?
\end{problem}

Brenti and Welker \cite{BW} proved that, if $P$ is a Cohen--Macaulay CW-poset of Boolean type,
then the $h$-polynomial of $\Delta_P$  has only real zeros.
This implies that its $f$-polynomial also has only real zeros, and therefore its $f$-vector is unimodal.

\subsection*{On flag $f$-vectors}
Theorem \ref{main4} gives sharp upper bounds of flag $f$-vectors of Gorenstein* posets
for a fixed rank generating function.
However, we do not have an answer to the following problem.

\begin{problem}
\label{7.1}
Find sharp upper bounds of the flag $f$-vectors of (Cohen--Macaulay) CW-posets
for a fixed rank generating function.
\end{problem}

More strongly, the next problem would be of great interest.

\begin{problem}
\label{7.2}
Characterize all possible flag $f$-vectors of (Cohen--Macaulay) CW-posets.
\end{problem}

The flag $f$-vectors of Gorenstein* posets of rank at most $4$ were characterized in \cite{MN}.
Consider this fact, we think that Problem \ref{7.2} will be tractable at least for CW-posets of rank $3$.

\subsection*{On $\cc\dd$-indices}
The proof of Theorem \ref{6.5} says that if $P$ is a Cohen--Macaulay quasi CW-poset of rank $n$,
then, for any partition $S=\{a\} \cup T \in \mathcal A_n$ with $a=\min S$, one has
$\alpha_S(P) \leq \alpha_{\{a\}}(P) \alpha_T(P)$.
Moreover, the same argument proves $\alpha_S(P) \leq \alpha_{T_1}(P) \alpha_{T_2}(P)$
for any partition $S=T_1 \cup T_2$ with $\max T_1 < \min T_2$.
We suggest the following conjecture which generalize this property.

\begin{conjecture}
\label{7.3}
Let $P$ be a Cohen--Macaulay CW-poset (or a Gorenstein* poset) of rank $n$ and $S \in \mathcal A_n$.
If $S=T_1 \cup T_2$ is a partition of $S$ then
$\alpha_S(P) \leq \alpha_{T_1}(P) \alpha_{T_2}(P)$.
\end{conjecture}

\end{document}